\documentclass[12pt,psamsfonts]{amsart}
\usepackage{amsmath,amsthm,amsfonts,amssymb}
\usepackage{eucal}
\usepackage{graphicx}
\usepackage[all,knot]{xy}
\xyoption{arc}

\newcommand{\e}{{\epsilon}}

\newcommand{\CC}{\mathcal{C}}

\newcommand{\lan}{\langle}
\newcommand{\ra}{\rangle}

\newcommand{\F}{\mathcal{F}}

\newcommand{\Rep}{{\rm Rep}}

\newcommand{\la}{\lambda}

\newcommand{\g}{\mathfrak{g}}
\newcommand{\N}{\mathbb{N}}
\newcommand{\Q}{\mathbb Q}
\newcommand{\R}{\mathbb{R}}

\DeclareMathOperator{\FPdim}{FPdim}

\newcommand{\mL}{\mathcal{L}}
\newcommand{\mN}{\mathcal{N}}
\DeclareMathOperator{\U}{U}
  \DeclareMathOperator{\Diag}{Diag}
 
\DeclareMathOperator{\Hom}{Hom}

\DeclareMathOperator{\SL}{SL}\DeclareMathOperator{\SU}{SU}\DeclareMathOperator{\PSL}{PSL}
\DeclareMathOperator{\Vect}{Vect}

\newcommand{\C}{\mathbb C}
\newcommand{\mC}{\mathcal{C}}

\newcommand{\mH}{\mathcal{H}}
\newcommand{\mZ}{\mathcal{Z}}
\newcommand{\mK}{\mathcal{K}}
\newcommand{\Z}{\mathbb Z}
\newcommand{\tS}{\tilde{S}}
\newcommand{\ts}{\tilde{s}}
\newcommand{\ot}{\otimes}
\newcommand{\B}{\mathcal{B}}

\numberwithin{equation}{section}

\newtheorem{theorem}{Theorem}[section]

\newtheorem{conj}[theorem]{Conjecture}

\newtheorem{prop}[theorem]{Proposition}
\theoremstyle{definition}

\newtheorem{remark}[theorem]{Remark}

\newtheorem{ex}[theorem]{Example}
\newtheorem{definition}[theorem]{Definition}

\begin{document}
\title[On classification of modular categories]
{On classification of modular tensor categories}

\author{Eric Rowell}
\email{rowell@math.tamu.edu}
\address{Department of Mathematics\\
    Texas A\&M University \\
    College Station, TX 77843\\
    U.S.A.}

\author{Richard Stong}
\email{stong@ccrwest.org}
\address{Center for Communications Research\\
    4320 Westerra Court\\
    San Diego, CA 92121-1969\\
    U.S.A.}

\author{Zhenghan Wang}
\email{zhenghwa@microsoft.com}
\address{Microsoft Station Q\\CNSI Bldg Rm 2237\\
    University of California\\
    Santa Barbara, CA 93106-6105\\
    U.S.A.}

\thanks{The first author is partially supported by NSA grant H98230-08-1-0020.
The second and third authors are partially supported by NSF FRG grant DMS-034772.  The third author likes to thank
Nick Read for his insightful comments on earlier versions.}

\begin{abstract}

We classify all unitary modular tensor categories (UMTCs) of rank
$\leq 4$. There are a total of $35$ UMTCs of rank $\leq 4$ up to ribbon tensor equivalence.
Since the distinction between the modular $S$-matrix $S$ and $-S$ has both
topological and physical significance,
so in our convention there are a total of $70$ UMTCs of rank $\leq 4$.
In particular, there are two trivial UMTCs with $S=(\pm1)$. Each such UMTC can be obtained from
$10$ non-trivial prime UMTCs by direct product, and some symmetry
operations. Explicit data of the $10$ non-trivial prime UMTCs are given in Section $5$.
Relevance of UMTCs to topological quantum computation and various
conjectures are given in Section $6$.

\end{abstract}
\maketitle

\section{Introduction}

\label{s:intro}

A modular tensor category (MTC) in the sense of V.~ Turaev
determines uniquely a (2+1)-topological quantum field theory
(TQFT) \cite{Turaev} (a seemingly different definition appeared in
\cite{MS1}.)  The classification of MTCs is motivated by the
application of MTCs to topological quantum computing
\cite{Freedman1}\cite{Kitaev1}\cite{FKW}\cite{FLWu}\cite{FKLW}\cite{Preskill},
and by the use of MTCs in developing a physical theory of
topological phases of matter
\cite{Wilczek}\cite{MR}\cite{FNTW}\cite{Kitaev2}\cite{Wang}\cite{LW}\cite{DFN}.
G.~ Moore and N.~ Seiberg articulated the viewpoint that rational
conformal field theory (RCFT) should be treated as a
generalization of group theory \cite{MS2}.  The algebraic content
of both RCFTs and TQFTs is encoded by MTCs. Although two seemingly
different definitions of MTCs were used in the two contexts
\cite{MS1}\cite{Turaev}, the two notions are essentially
equivalent: an MTC in \cite{MS1} consists of essentially the basic
data of a TQFT in \cite{Walker}.  The theory of MTCs encompasses
the most salient feature of quantum mechanics in the tensor
product: superposition. Therefore, even without any applications
in mind, the classification of MTCs could be pursued as a quantum
generalization of the classification of finite groups.

There are two natural ways to organize MTCs: one by fixing a pair
$(G,\lambda)$, where $G$ is a compact Lie group, and $\lambda$ a
cohomology class $\in H^4(BG;\Z)$; and the other by fixing the
rank of an MTC, i.e. the number of isomorphism classes of simple
objects. If a conjecture of E.~Witten were true, then every MTC
would come from a Chern-Simons-Witten (CSW) TQFT labelled by a
pair $(G,\lambda)$ \cite{Witten}\cite{MS1}\cite{HRW}.
Classification by fixing a compact Lie group $G$ has been carried
out successfully for $G$=finite groups \cite{DW}\cite{FQ}, $G=T^n$
torus \cite{Ma}\cite{BM}, and $G=A,B,C,D$ simple Lie groups
\cite{FK}\cite{KW}\cite{TuWe}. In this paper, we will pursue the
classification by fixing the rank. This approach is inspired by
the study of topological phases of matter and topological quantum
computing. Another reason is that we have evidence that there
might be exotic (2+1)-TQFTs other than CSW theories \cite{HRW}.

Topological phases of matter are like artificial elements.  The
only known topological phases of matter are fractional quantum
Hall liquids: electron systems confined on a disk immersed in a
strong perpendicular magnetic field at extremely low temperatures
\cite{Wilczek}\cite{DFN}.  Electrons in the disk, pictured classically as
orbiting inside concentric annuli around the origin, organize
themselves into some topological order \cite{Wen}\cite{WW1}\cite{WW2}.
Therefore, the classification of topological phases of matter
resembles the periodic table of elements.  The periodic table does
not go on forever, and simpler elements are easier to find. The
topological quantum computing project is to find MTCs in Nature,
in particular those with non-abelian anyons. Therefore, it is
important that we know the simplest MTCs in a certain sense
because the chance for their existence is better.

There is a hierarchy of structures on a tensor category: rigidity,
pivotality, sphericity.  We will always assume that our category
is a fusion category: a rigid, semi-simple, $\C$-linear monoidal
category with finitely many isomorphism classes of simple objects,
and the trivial object is simple.  It has been conjectured that
every fusion category has a pivotal structure \cite{ENO}.
Actually, it might be true that every fusion category is
spherical. Another important structure on a tensor category is
braiding.  A tensor category with compatible pivotal and braiding
structures is called ribbon. In our case a ribbon category is
always pre-modular since we assume it is a fusion category.  For
each structure, we may study the classification problem.  The
classification of fusion categories by fixing the rank has been
pursued in \cite{Ostrik1}\cite{Ostrik2}. Since an MTC has
considerably more structures than a fusion category, the
classification is potentially easier, and we will see that this is
indeed the case in Sections \ref{rank23} and \ref{rankIV}. The
advantage in the MTC classification is that we can work with the
modular $S$ matrix and $T$ matrix to determine the possible fusion
rules without first solving the pentagon and hexagon equations.
For the classification of MTCs of a given rank, we could start
with the infinitely many possible fusion rules, and then try to
rule out most of the fusion rules by showing the pentagon
equations have no solutions. However, pentagon equations are
notoriously hard to solve, and we have no theories to practically
determine when a solution exists for a particular set of fusion
rules (Tarski's theorem on the decidability of the first-order
theory of real numbers provides a logical solution). So being able
to determine all possible fusion rules without solving the
pentagon equations greatly simplifies the classification for MTCs.
As shown in \cite{HH}, all structures on an MTC can be formulated
as polynomial equations over $\Z$. Hence the classification of MTC
is the same as counting points on certain algebraic varieties up to equivalence.
But all the data of an MTC can be presented over certain finite degree
Galois extension of $\Q$, probably over abelian Galois extension
of $\Q$ if normalized appropriately. Therefore, the classification
problem is closer to number theory than to algebraic geometry. The
argument in Sections \ref{rank23} and \ref{rankIV} is basically
Galois theory plus elementary yet complicated number theory. To
complete the classification, we need to solve the pentagons and
hexagons given the fusion rules. A significant complication comes from the choices of
bases of the Hom spaces when solving the pentagon equations. The
choices of basis make the normalization of $6j$ symbols into an
art: so far no computer programs are available to solve pentagons
with a fusion coefficient $>1$, but one set of such fusion rules
is solved completely \cite{HH}. Currently, there are no theories
to count the number of solutions of pentagon equations for a given
set of fusion rules without solving the pentagons.  For unitary MTCs,
there is tension between two desirable normalizations for $6j$
symbols: to make the $F$ matrix unitary, or to present all data of
the theory in an abelian Galois extension of $\Q$. For the
Fibonacci theory, unitarity of the $F$ matrix and abelianess of
the Galois extension of $\Q$ cannot be achieved simultaneously,
but with different $F$ matrices, each can be obtained separately
\cite{FW}. This is the reason that we will only define the Galois
group of a modular fusion rule and a modular data, but not the
Galois group of an MTC.

The main result of this paper is the classification of MTCs with
rank=$2,3$ and unitary MTCs of rank=$4$. The authors had obtained
the classification of all unitary MTCs of rank $\leq 4$ in 2004
\cite{Wang}. The delay is related to the open finiteness
conjecture:  \emph{There are only finitely many equivalence classes of MTCs for
any given rank.}
By Ocneanu rigidity the conjecture is equivalent to: \emph{There are only
finitely many sets of fusion rules for MTCs of a given rank.} Our
classification of MTCs of rank $\leq 4$ supports the conjecture.
We also listed all quantum group MTCs up to rank $\leq 12$ in
Section \ref{realization}.  Two
well-known constructs of MTCs are the quantum group method, and
the quantum double of spherical tensor categories or the Drinfeld
center. The quantum double is natural for MTCs from subfactor
theory using Ocneanu's asymptotic inclusions \cite{Evans}. It
seems that this method might produce exotic MTCs in the sense of
\cite{HRW}.

Our main technique is Galois theory. Galois theory was introduced
into the study of RCFT by J.~de Boer and J.~ Goeree \cite{dBG},
who considered the Galois extension $K$ of $\Q$ by adjoining all
the eigenvalues of the fusion matrices.  They made the deep
observation that the Galois group of the extension $K$ over $\Q$
is always abelian. This result was extended by A.~ Coste and T.~
Gannon who used their extension to study the classification of
RCFTs \cite{CG}.  Fusion
rules of an MTC are determined by the modular $S$-matrix through
Verlinde formulas. It follows that the Galois extension $K$ is the
same as adjoining to $\Q$ all entries of the modular $\tS$ matrix.
When a Galois group element applies to the $\tS$ matrix entry-wise,
this action is a multiplication of $\tS$ by a signed
permutation matrix, which first appeared in \cite{CG}.
It follows that the entries of the $\tS$ matrix are the same up
to signs if they are in the same orbit of a Galois group element.
For a given rank $\leq 4$, this allows us to determine all
possible $\tS$-matrices, therefore, all possible fusion rules.

Note that the Galois group of a modular data does not change the
fusion matrices, but it can change a unitary theory into a
non-unitary theory.  For example, the Galois conjugate of the
Fibonacci theory is the Yang-Lee theory, which is non-unitary. We
might expect that for each modular data, one of its Galois
conjugates would be realized by a unitary MTC. This is actually
false. For example, take a rank=$2$ modular data with
$\tS=\begin{pmatrix}
   1& -1\\
   -1 & -1
 \end{pmatrix}
 $, and $T=\begin{pmatrix}
   1& 0\\
   0& i
 \end{pmatrix}
 $.  No Galois actions can change the $\tS$ matrix, hence the
 quantum dimension of the non-trivial simple object from $-1$ to $1$, though the
 same fusion rules can be realized by a unitary theory: the semion theory.
 Reference \cite{rowell06} contains a set of fusion rules which has
non-unitary MTC realizations, but has no unitary realizations at
all.

The paper is organized as follows.  In Section \ref{galoistheory},
we study the implications of the Verlinde formulas using Galois
theory.  In Sections \ref{rank23} and \ref{rankIV}, we determine
all self-dual modular $\tS$ matrices of modular symbols of
rank=$2,3$, and unitary ones for rank=$4$.  Rank=$2$ is known to experts, and
rank=$3$ fusion rules have been previously classified \cite{CP}.  For modular data,
Theorems \ref{Srank2} and \ref{Srank3} can also be deduced from
\cite{Ostrik1}\cite{Ostrik2}.
 In Section \ref{realization}, we determine all UMTCs of rank $\leq 4$.
 In Section \ref{conjecture}, we discuss some open
questions about the structure and application of MTCs.  In
Appendix \ref{nonselfdual}, together with S.~Belinschi, we
determine all non-self dual unitary modular data of rank $\leq 4$.

We summarize the classification of all rank $\leq 4$ unitary MTCs
into Table \ref{rank4table}. There are a total of $70$ unitary
MTCs of rank $\leq 4$ (a total of $35$ up to ribbon tensor equivalence).
The count is done in Section \ref{count}.  Each such UMTC can be obtained
from $10$ non-trivial prime UMTCs by direct product, and some
symmetry operations.  The $10$ non-trivial prime UMTCs are the
semion MTC, the Fibonacci MTC or $(A_1,3)_{\frac{1}{2}}$, the
$\Z_3$ MTC, the Ising MTC, the $(A_1,2)$ MTC, the even half of an
$SU(2)$ MTC at level 5 or $(A_1,5)_{\frac{1}{2}}$, the $\Z_4$ MTC,
the toric code MTC, the $(D_4,1)$ MTC, and the even half of an
$SU(2)$ MTC at level $7$ or $(A_1,7)_{\frac{1}{2}}$. Their
explicit data are listed in Section \ref{explicitdata}. Out of the
$10$ non-trivial prime UMTCs, $9$ are quantum group categories for
a simple Lie group: the semion=$SU(2)_1$, the Fibonacci=$(G_2)_1$,
the $\Z_3$=$SU(3)_1$, the Ising=complex conjugate of $(E_8)_2$,
the $(A_1,2)$=$SU(2)_2$, the $\Z_4=SU(4)_1$, the toric
code$=Spin(16)_1$, the $(D_4,1)=Spin(8)_1$, and the
$(A_1,7)_{\frac{1}{2}}$=complex conjugate of $(G_2)_2$.
The Ising MTC and the $SU(2)_2$
MTC have the same fusion rules, but the Frobenius-Schur indicators
of the non-abelian anyon $\sigma$ are $+1,-1$, respectively.  The
toric code MTC and the $Spin(8)_1$ MTC have the same fusion rules,
but the twists are $\{1,1,1,-1\}$, and $\{1,-1,-1,-1\}$,
respectively.  We choose $q=e^{\frac{\pi i}{\ell}}$ in the quantum
group construction. In the Ising case, it is the $q=e^{-\frac{\pi
i}{\ell}}$ theory for $E_8$ at level=$2$. For notation and more
details, see Section \ref{explicitdata}.  We do not know how to construct
$(A_1,5)_{\frac{1}{2}}$ by cosets of quantum group categories.

 The information for each rank is contained in one row of Table \ref{rank4table}.
 Each box contains
information of the MTCs with the same fusion rules. The center
entry in a box denotes the realization of the fusion by a quantum
group category or their products.  We also use $Fib$ to denote the
Fibonacci category $(A_1,3)_{\frac{1}{2}}$. The upper left corner
has either $A$ or $N$, where $A$ means that all anyons are
abelian, and $N$ that at least one type of anyons is non-abelian.
The right upper corner has a number which is the number of
different unitary theories with that fusion rule.  If the lower
right corner has a $U$, it means that at least one type of anyons
has universal braiding statistics for topological quantum
computation.  The detailed information about which anyon is
abelian or non-abelian, universal or non-universal is given in
Section \ref{compiling}.  It is worth
noticing that the list of all fusion rules up to rank=$4$ agrees
with the computer search for RCFTs in \cite{GK}.  We believe this
continues to be true for rank=$5$.  The rank$=6$ list in
\cite{GK} is not complete.

Finally, we comment on the physical realization of UMTCs. The
existence of abelian anyons in $\nu=\frac{1}{3}$ FQH liquids is
established theoretically with experimental support, while
non-abelian anyons are believed to exist at the $\nu=\frac{5}{2}$
and $\nu=\frac{12}{5}$ plateaus (see \cite{DFN} and the
references therein). Current experimental effort is focused on FQH
liquids at $\nu=\frac{5}{2}$.  But the fermionic nature of
electrons complicates direct application of MTCs to FQH liquids
because only anyonic properties of bosonic systems can be
described fully by MTCs.  In other words, we need a refined
theory, e.g. a spin MTC, to describe a fermionic system
\cite{DW}\cite{BM}.

\begin{table}\caption{Unitary MTCs of rank$\leq4$}\label{rank4table}
\begin{tabular}{|lcr|lcr|lcr|lcr|lcr|}

\hline
 & & & A& & 2& &              &  & &            & & & & \\
 & & & &\;\;\;\;\bf{1} \;\;\;\;& &   &       &   & &            & & & & \\
 & & & &    & &   &              &   & &            & & & & \\
\hline
 & & &A&   &4&  &              &  &N&            &4& & & \\
 & & & &\;\;\;\;\bf{$\Z_2$} \;\;\;\;& &   &  &   & &\bf{$(A_1,3)_{\frac{1}{2}}$}& & & & \\
& & & &    & &   &              &   & &            &U& & & \\
\hline
 & & &A&   &4& N &              &  16&N&            &4& & & \\
 & & & &\;\;\;\;\bf{$\Z_3$}\;\;\;\;& &   & \bf{$(A_1,2)$} &   & &\bf{$(A_1,5)_{\frac{1}{2}}$}& & & & \\
& & & &    & &   &              &   & &            &U& & & \\
\hline
A& &10&A&   &8& N &              &  8&N&            &4&N& & 6\\
 &\bf{$\Z_2\times \Z_2$}& & &\;\;\;\;\bf{$\Z_4$}\;\;\;\;& &   & \bf{$(A_1,3)$} &   & &\bf{$(A_1,7)_{\frac{1}{2}}$}& & &\bf{$Fib\times Fib$}& \\
& & & &    & &   &              &   U& &            &U& & &U \\
 \hline
\end{tabular}
\end{table}

\section{Galois theory of fusion rules}\label{galoistheory}

In this section, we study the implication of Verlinde formulas for
fusion rules of MTCs.  For more related
discussion, see the beautiful survey \cite{Gannon}.

\begin{definition}\label{fusionrule}

\begin{enumerate}

\item A rank=$n$ label set is a finite set $\mL$ of $n$ elements
with a distinguished element, denoted by $0$, and an involution
$\hat{}: \mL\rightarrow \mL$ such that $\hat{0}=0$.  A label $i\in
\mL$ is self dual if $\hat{i}=i$.

The charge conjugation matrix is the $n\times n$ matrix
$C=(\delta_{i\hat{j}})$.  Note that $C$ is symmetric and
$C^2=I_n$, the $n\times n$ identity matrix.

\item A rank=$n$ modular fusion rule is a pair $(\mN;\tS)$, where
$\mN$ is a set of $n$ $n\times n$ matrices $N_i=(n_{i,j}^k)_{0\leq
j, k\leq n-1} $, indexed by a rank=$n$ label set $\mL$, with
$n_{i,j}^k\in \Q$ , and $\tS=(\ts_{ij})_{0\leq i,j\leq n-1}$ is an
$n\times n$ matrix satisfying the following: 

\vspace{.1in}

\begin{enumerate}

\item $\ts_{00}=1, \ts_{i,\hat{j}}=\overline{\ts_{i,j}}$, and all
$\ts_{i,0}$'s are non-zero;

\item  If we let $D=\sqrt{\sum_{i=0}^{n-1}\tilde{s}_{i,0}^2}$,
then $S=\frac{\tS}{D}$ is a symmetric, unitary matrix.

\end{enumerate}

Furthermore, the matrices $N_i$ in $\mN$ and $\tS$ are related by
the following:

\begin{equation}\label{Verlinde}
N_i \tS=\tS \Lambda_i
\end{equation}
for all $i\in \mL$, where
$\Lambda_i=(\delta_{ab}\lambda_{ia})_{n\times n}$ is diagonal, and $\lambda_{ia}=\frac{\ts_{ia}}{\ts_{0a}}$.

The identities (\ref{Verlinde}) or equivalently the Verlinde
formulas (\ref{verlindeformula}) below imply many symmetries among
$n_{i,j}^k$: $n_{0,j}^k=\delta_{jk},
n_{i,j}^k=n_{j,i}^k=n_{\hat{i},\hat{j}}^{\hat{k}}=n_{i,\hat{k}}^{\hat{j}}$.

The matrix $N_i$ will be called the $i$th fusion matrix.  From
identities (\ref{Verlinde}), the diagonal entries in $\Lambda_i$ are
the eigenvalues of $N_i$, and the columns of $\tS$ are the
corresponding eigenvectors.  The non-zero number $D$ will be
called the total quantum order, $d_i=\tilde{s}_{i0}$ the quantum
dimension of the $i$-th label, and $D^2=\sum_{i=0}^{n-1}d_i^2$ the
global quantum dimension.

\item A rank=$n$ modular symbol consists of a triple $(\mN;S,T)$.
The pair $(\mN;\frac{S}{s_{00}})$ is a rank=$n$ modular fusion
rule with all $n_{i,j}^k\in \N=\{0,1,2,\cdots\}$ (here $s_{00}$ is
the (0,0)-entry of the unitary matrix $S=(s_{ij})_{0\leq i,j\leq
n-1}$), and the $n\times n$ matrix
$T=(\delta_{ab}\theta_a)_{n\times n}$ is diagonal, and $\theta_0=1$. Furthermore,
$S$ and $T$ satisfy

(i) $(ST)^3=(D_{+}s_{00})S^2$;

(ii) $S^2=C$;

(iii) $\theta_i \in \U(1)$ and $\theta_{\hat{i}}=\theta_i$ for
each $i$;

where $D_{\pm}=\sum_{i=0}^{n-1}\theta_i^{\pm 1} d_i^2$ .  The
following identity can be deduced:
\begin{equation}\label{balance}
D_+D_{-}=D^2.
\end{equation}
The complex number $\theta_i$ will be called the twist of the
$i$th label.  Note that $s_{00}$ might be $-\frac{1}{D}$.  A
modular symbol is called unitary if each quantum dimension $d_i$
is the Frobenius-Perron eigenvalue of the corresponding fusion
matrix $N_i$.  In particular, the quantum dimensions $d_i$'s are
positive real numbers $\geq 1$.

\item A modular symbol $(\mN;S,T)$ is called a modular data if there is an MTC whose
fusion rules, modular $\tilde{S}$-matrix, and $T$-matrix are given by
$\mN, \frac{S}{s_{00}}, T$ of the modular symbol.

\item Let $\Lambda=\{\lambda_{ij}\}_{i,j\in \mL}$ for a rank=$n$
modular fusion rule, and let $K=\Q(\lambda_{ij}), i,j\in \mL$ be the
Galois extension of $\Q$. Then the Galois group $G$ of the Galois
field $K$ over $\Q$ is called the Galois group of the modular fusion rule.

\end{enumerate}

\end{definition}

We are interested in searching for $n+2$ tuples
$(N_0,\cdots,N_{n-1};\tilde{S},T)$ related in the correct fashion.
We will index the rows and columns of matrices by
$0,1,\cdots,n-1$. Since $N_i \tilde S=\tS \Lambda_i$, the columns
of $\tilde S$ must be eigenvectors of $N_i$ with eigenvalues
$\lambda_{i,0}$, $\lambda_{i,1}$, $\cdots$, and $\lambda_{i,n-1}$,
respectively. Looking at the first entries of these columns and of
$N_i\tilde S$, and using the only non-zero $1$ of the first row of
$N_i$, we see that $\lambda_{i,0}=d_i$, and
$d_j\lambda_{i,j}=\tilde s_{i,j}$.  It follows that $K$ is the
same as $\Q(\tilde{s}_{ij}), i,j\in \mL$. Since $\tilde S$ is
symmetric, we see that for $i\ne j$ we have
$d_j\lambda_{i,j}=d_i\lambda_{j,i}$, and $\tilde s_{i,j} =
d_i\lambda_{j,i}=d_j\lambda_{i,j}$ for all $i$ and $j$. Let
$n_{i,j}^k$ denote the $(j,k)$ entry of $N_i$. Since
$$N_i = \frac{1}{D^2} \tilde S \left(\begin{matrix} \lambda_{i,0} & 0 & \cdots &
0\cr 0 & \lambda_{i,1} & \cdots & 0\cr \cdots & \cdots & \cdots &
\cdots \cr 0 & 0 & \cdots & \lambda_{i,n-1}\end{matrix}\right)
\tS^{\dagger},$$ we compute for $0\le j,k\le n-1$
\begin{equation}\label{verlindeformula}
n_{i,j}^k =\sum_{m=0}^{n-1} \frac{\tilde s_{i,m} \ts_{j,m}
\overline{\ts_{k,m}}}{d_m} D^{-2} = \sum_{m=0}^{n-1} \lambda_{i,m}
\lambda_{j,m} \overline{\lambda_{k,m}} \frac{d_m^2}{D^2}.
\end{equation}

The fusion matrices can also be described equivalently by fusion
algebras.  For a rank=$n$ fusion rule, each label $i$ is
associated with a variable $X_i$. Then the fusion ring $R$ is the
free abelian ring $\Z[X_0,\cdots, X_{n-1}]$ generated by $X_i$'s
modulo relations (called fusion rules) $X_i
X_j=\sum_{i=0}^{n-1}n^k_{i,j}X_k.$ The fusion algebra will be
$F=R\otimes_{Z}K$, where $K$ is the Galois field of the fusion
rules above.  We may replace $K$ by $\C$. If the modular fusion
rule is realized by an MTC, then $X_i$ is an equivalence class of
simple objects, and the multiplication $X_iX_j$ is just the tensor
product.

There are modular symbols that are not modular data.

\begin{ex}\label{isingex}

Take the following
$$S=\frac{1}{2} \left(\begin{matrix} 1 & \sqrt{2} & 1\cr
\sqrt{2} & 0 & -\sqrt{2}\cr 1 & -\sqrt{2} &
1\end{matrix}\right),$$ and $T=\Diag(1,\theta,-1)$.  The fusion
matrices $N_i$ are determined by the formulas
(\ref{verlindeformula}), hence are independent of $\theta$.  They
are the same as those of the Ising MTC in Section
\ref{explicitdata}. Therefore, for any $\theta \in \U(1)$, we get
a modular symbol. But only when $\theta$ is a $16$th root of
unity, do we have modular data.

\end{ex}

Very likely the modular symbol of an MTC determines the MTC , and we do not know
when a modular symbol becomes a modular data.

\begin{prop}\label{modularprop}

If $(\mN;S,T)$ is a modular data, then we have:

\begin{enumerate}

\item $\theta_i\theta_j s_{ij}=\sum_k n_{\hat{i}j}^k
s_{k0}\theta_k$;

\item $\prod_j
\theta_j^{A_{ij}}=\theta_i^{\frac{4}{3}\sum_jA_{ij}}$;

where $A_{ij}=2n_{i\hat{i}}^jn_{ij}^i+n_{ii}^jn_{j\hat{i}}^i$;

\item Let $\nu_k=\frac{1}{D^2}\sum_{i,j\in \mL}n_{k,j}^i d_id_j\frac{\theta_i^2}{\theta_j^2}$,
then $\nu_k$ is $0$ if $k\neq \hat{k}$, and is $\pm 1$ if $k=\hat{k}$.  $\nu_k$ is called
the Frobenius-Schur indicator of $k$.

\item $D_{+}s_{00}=e^{\frac{\pi i c}{4}}$ for some $c\in \Q$.  The
rational number $c$ mod $8$ is called the topological central
charge of the modular data.

\end{enumerate}

\end{prop}

\begin{proof}
For (1), see \cite[Eq. (3.1.2)]{BK} on page 47.  For (2), it is
\cite[Theorem 3.1.19]{BK} found on page 57.  Formula (3) from
\cite{Bantay1} for RCFTs can be  generalized to MTCs. (4)
follows from Theorem \ref{deboergoereetheorem}.
\end{proof}

Proposition \ref{modularprop} (2) implies that the $\theta_i$ are
actually roots of unity of finite order, which is often referred
to as \emph{Vafa's Theorem}. But from example \ref{isingex}, we
know that this is not true for general modular symbols, in
particular $\Q(\theta_i)$ might not be algebraic for modular
symbols. This leads to:
\begin{definition}

Given a modular data $(\mN; S,T)$, let $\mK_N$ be the Galois
field $\Q(\tilde{s}_{ij},D,\theta_i),i,j\in \mL$.  Then the Galois
group of $\mK_N$ over $\Q$ will be called the Galois group of the
modular data.

\end{definition}

\begin{theorem}\label{deboergoereetheorem}

\begin{enumerate}
\item (de Boer-Goeree theorem): The Galois group of a modular
fusion rule is abelian.

\item The Galois group of a modular data is abelian.
\end{enumerate}
\end{theorem}

By the Kronecker-Weber theorem, there is an integer $m$ such that
$\mK_N\subset \Q(\zeta_m)$, where $\zeta_m=e^{\frac{2\pi i}{m}}$.
The smallest such $m$ for $\mK_N$ is called the conductor of
$\mK_N$, and the order of $T$ always divides $N$ (we intentionally
build $N$ into the notation $\mK_N$). The Galois group of
$\Q(\zeta_N)$ is the cyclic group of units $l$ such that
$\gcd(l,N)=1$.  Each $l$ acts on $\mK_N$ as the Frobenius map
$\sigma_l: \zeta_N\rightarrow \zeta_N^l$. Consequently,
$\sigma_l(T)=T^l$ and $\sigma_l(S)=S\tilde{P_{\sigma}}$, where the
signed permutation matrix $\tilde{P_\sigma}$ corresponds to the
Galois element $\sigma$ in the Galois group of the modular fusion
rule.

It is known that the fusion algebra of a rank=$n$ MTC is
isomorphic to the function algebra of $n$ points.  A Galois group
element $\sigma$ of the associated modular fusion rule induces an
isomorphism of the fusion algebra.  It follows that $\sigma$
determines a permutation of the label set.  When we have only a
modular fusion rule, the two algebra structures on the fusion
algebra a priori might not be isomorphic to each other. But still
a Galois group element of the modular fusion determines a
permutation of the label set and the de Boer-Goeree theorem holds.
Actually what we are using in this paper are identities among
modular $\tS$ entries up to some parity signs
$\epsilon_{i,\sigma}=\pm 1$ associated to each Galois element
$\sigma$.  Such parity signs first appeared in \cite{CG} for
Galois automorphisms of $\Q(\lambda_{i,j},D)$. 

First we note the following easy, but very useful fact that the
ordered set of eigenvalues of $N_i$ determines the label $i$
uniquely.

\begin{prop}\label{determinelabel}

There do not exist indices $j\ne k$ such that
$\lambda_{i,j}=\lambda_{i,k}$ for all $i$ for any modular fusion
rule $(\mN;\tS)$.

\end{prop}

\begin{proof}

If there were such indices, then the dot product of rows $j$ and
$k$ of $\tilde S$ would be $D^2=\sum_{i=0}^{n-1} |\ts_{i,j}|^2
>0$, a contradiction.

\end{proof}

Except (5), the following theorem is contained in \cite{CG}.

\begin{theorem}\label{Galoissymmetry}

Let $G$ be the Galois group of a rank=$n$ modular fusion rule $(\mN;\tS)$.
Then

\begin{enumerate}

\item the simultaneous action of the Galois group $G$ on the set
$\Lambda=\{\lambda_{ij}\}$ gives an injective group homomorphism
$\iota: G\rightarrow S_n$, where $S_n$ is the permutation group of
$n$ letters; for $\sigma \in G$, $\iota(\sigma)(i)$ is the
associated element in $S_n$;

\item for any $\sigma \in G$, the matrix
$\tilde{P_{\sigma}}=d_{\sigma(0)}\tS^{-1}\sigma(\tS)$ is a signed
permutation matrix; furthermore, the map $\sigma\rightarrow
\tilde{P_{\sigma}}$ gives a group homomorphism from $G$ to the
signed permutation matrices modulo $\pm 1$ which lifts $\iota$;

\item for each $\sigma \in G$, there are $\e_{i,\sigma}=\pm 1$
such that
\begin{equation}\label{actionidentity1}
\sigma(\tilde{s}_{j,k})=\frac{1}{d_{\sigma(0)}}\e_{\sigma(k),\sigma}\tilde{s}_{j,\sigma(k)}.
\end{equation}
Moreover,
\begin{equation}\label{actionidentity}
\tilde{s}_{j,k}=\e_{\sigma(j),\sigma}\e_{k,\sigma}\tilde{s}_{\sigma(j),\sigma^{-1}(k)},
\end{equation}
and
\begin{equation}\label{inverse}
\e_{\sigma^{-1}(k),\sigma^{-1}}=\e_{\sigma(0),\sigma}\e_{0,\sigma}\e_{k,\sigma};
\end{equation}

\item the Galois group $G$ is abelian;

\item If $n$ is even, then
$\prod_{i=0}^{n-1}\e_{i,\sigma}=(-1)^{\sigma}$.  If $n$ is odd,
then $D\in K$, and $\sigma(D)=\e_{\sigma}\cdot
\frac{D}{d_{\sigma(0)}}$, where $\e_{\sigma}=\pm 1$. We have
$\prod_{i=0}^{n-1}\e_{i,\sigma}=\e_{\sigma} \cdot (-1)^{\sigma}$.


\end{enumerate}

\end{theorem}

We are going to use $\sigma$ for both the element of the Galois
group $G$ and its associated element of $S_n$.  When $\sigma \in
G$ applies to a matrix, $\sigma$ applies entry-wise.

\begin{proof}

Let $K={\mathbb Q}[\{\lambda_{i,j}\}_{0\le i,j\le n-1}]$ be the
Galois extension of ${\mathbb Q}$ generated by the eigenvalues of
all the $N_i$ and let $G$ be the associated Galois group as above.
The action of $G$ on the eigenvalues gives an injection $G\to
S_{n}\times S_{n}\times \cdots \times S_{n}$, where there are
$n-1$ factors. Note that we have not assumed the $N_i$ have
distinct eigenvalues, therefore this map is not necessarily unique
and is not necessarily a group homomorphism. This is not a problem
as we will resolve the ambiguity shortly. Just fix one such map
for now. Let $(\sigma_1,\sigma_2,\cdots, \sigma_{n-1})$ denote the
image of $\sigma\in G$ under this injection. Note that a priori,
there is no relationship between the $\sigma_i$. Let $\Lambda_i$
be the diagonal matrix with diagonal entries $\lambda_{i,j}$, so
$$N_i=\tilde S \Lambda_i \tilde
S^{-1}.$$ Let $P_{\sigma_i}=(\delta_{i=\sigma_i(j)})_{0\le i,j\le
n-1}$ be the permutation matrix corresponding to $\sigma_i$. Since
$\sigma(\lambda_{i,j})=\lambda_{i,\sigma_i(j)}$, we have
$\sigma(\Lambda_i)=P_{\sigma_i}^{-1} \Lambda_i P_{\sigma_i}$.
Since $N_i$ is rational we have
$$\tilde S \Lambda_i \tilde S^{-1} = N_i =\sigma(N_i) = \sigma(\tilde S)
P_{\sigma_i}^{-1} \Lambda_i P_{\sigma_i} \sigma(\tilde S)^{-1}.$$
Rewriting this gives
$$\Lambda_i [\tilde S^{-1} \sigma(\tilde S) P_{\sigma_i}^{-1}]=
[\tilde S^{-1} \sigma(\tilde S) P_{\sigma_i}^{-1}] \Lambda_i ,$$
Hence $B_{i,\sigma}=\tilde S^{-1} \sigma(\tilde S)
P_{\sigma_i}^{-1}$ commutes with $\Lambda_i$. It follows that
$B_{i,\sigma}$ is block diagonal, with blocks  corresponding to
the equal eigenvalues of $N_i$. In formulas, if the $(j,k)$ entry
of $B_{i,\sigma}$ is nonzero, then $\lambda_{i,j}=\lambda_{i,k}$.
Let $\tilde S^{-1} \sigma(\tilde S) = B_{i,\sigma} P_{\sigma_i} =
C_{\sigma}$. Note two facts, if the $(j,k)$ entry of $C_{\sigma}$
is nonzero, then the $(j,\sigma_i(k))$ entry of $B_{i,\sigma}$ is
nonzero and hence $\lambda_{i,j}=\lambda_{i,\sigma_i(k)}$. The
second fact is that $C_{\sigma}$ (as the notation suggests) does
not depend on $i$, only on $\sigma$. Suppose $C_{\sigma}$ has 2
nonzero entries in column $k$, say the $(j,k)$ and $(\ell,k)$
entries. Then
$\lambda_{i,j}=\lambda_{i,\sigma_i(k)}=\lambda_{i,\ell}$ for all
$i$, contradicting the Proposition \ref{determinelabel} above. If
a row or column of $C_{\sigma}$ is all zeroes, then
$\det(C_{\sigma})=0$, a contradiction. Hence $C_{\sigma}$ has
exactly one nonzero entry in every row and in every column. Thus
there is a unique permutation $\sigma\in S_n$ and a diagonal
matrix $B_{\sigma}$ such that $C_{\sigma}=B_{\sigma} P_{\sigma}$.
Note that we are now using $\sigma$ for both the element of the
Galois group and its associated element of $S_n$. Note that
$$C_{\sigma \sigma^{\prime}} =\tilde S^{-1} \sigma \sigma^{\prime}(\tilde
S) = \tilde S^{-1} \sigma(\tilde S) \sigma(\tilde S^{-1}
\sigma^{\prime}(\tilde S)) = C_{\sigma}
\sigma(C_{\sigma^{\prime}}).$$ From which it follows that the map
$G\to S_{n}$ is a group homomorphism. Thus we have proved that the
simultaneous action of the Galois group $G$ on the eigenvalues
$\lambda_{i,j}$ of $N_i$ for all $i$ gives an injective group
homomorphism $G\to S_{n}$.

Note that the squared length of column zero of $\tilde S$ is
$D^2=\sum_{i=0}^{n-1} d_i^2$ which must be equal to the squared
length of column $\sigma(0)$. Hence
$$D^2 = \sum_{i=0}^{n-1} d_{\sigma(0)}^2 \lambda_{i,\sigma(0)}^2 =
d_{\sigma(0)}^2 \sigma\left(\sum_{i=0}^{n-1}
\lambda_{i,0}^2\right)= d_{\sigma(0)}^2 \sigma(D^2).$$ Rewriting
gives
$$\sigma\left(\frac{1}{D^2}\right) = \frac{d_{\sigma(0)}^2}{D^2}.$$
It follows that $G$ acts in the same way on the quantities $\{
d_j/D^2\}$.  The Verlinde
formulas (\ref{verlindeformula}) encode the symmetry of the $N_i$ matrices,
and give us
the complete symmetry under interchanging the last $n-1$ $N_i$ and
simultaneously reordering the last $n-1$ rows and columns of all
matrices.  Thus $n_{i,j}^k$ is invariant under $G$ and hence is
necessarily rational if we define it first to be only in $\R$.

Transposing the identity $\tilde S^{-1} \sigma(\tilde
S)=C_{\sigma}$ and inverting this identity gives the two equations
$\sigma(\tilde S) \tilde S^{-1} = C_{\sigma}^T$ and
$$C_{\sigma}^{-1} = \sigma(\tilde S)^{-1} \tilde S =
\frac{D^2}{\sigma(D^2)} \sigma(\tilde S) \tilde S^{-1} =
d_{\sigma(0)}^2 C_{\sigma}^T.$$ Hence the matrices $d_{\sigma(0)}
C_{\sigma}$ and $d_{\sigma(0)} B_{\sigma}$ are orthogonal. Since
$B_{\sigma}$ is diagonal it follows that
$$B_{\sigma}= \frac{1}{d_{\sigma(0)}} \left(\begin{matrix} \e_{0,\sigma} &
0 & \cdots & 0\cr 0 & \e_{1,\sigma} & \cdots & 0\cr \cdots &
\cdots & \cdots & \cdots \cr 0 & \cdots & \cdots &
\e_{n-1,\sigma}\end{matrix}\right)$$ for some choices of
$\e_{i,\sigma}=\pm 1$. The map $\sigma\mapsto d_{\sigma(0)}
C_{\sigma}$ gives a group homomorphism from $G$ to the signed
permutation matrices modulo $\pm 1$ which lifts the homomorphism
$\iota$ of (1).

Rewrite the definition of $C_{\sigma}$ as $\sigma(\tilde S)=
\tilde S B_{\sigma} P_{\sigma}$. Picking out the $(j,k)$ entry, we
have
$\sigma(\tilde{s}_{j,k})=\frac{1}{d_{\sigma(0)}}\e_{\sigma(k),\sigma}\tilde{s}_{j,\sigma(k)}$.
 Moreover, since the left hand side is symmetric we get $\tilde S B_{\sigma}
P_{\sigma} = P_{\sigma}^{-1} B_{\sigma} \tilde S$. In coordinates
this condition becomes $\tilde s_{j,k}=\e_{k} \e_{\sigma(j)}
\tilde s_{\sigma(j),\sigma^{-1}(k)}$. Consider the action of $G$
on pairs $(j,k)$ defined by $\sigma\times (j,k)\mapsto
(\sigma(j),\sigma^{-1}(k))$. Then we see that $\vert \tilde
s_{j,k}\vert$ is constant on orbits of this action.

To see identity (\ref{inverse}), we apply $\sigma^{-1}$ to identity
(\ref{actionidentity1}) and compare with identity
(\ref{actionidentity}). Note that
$\ts_{\sigma^{-1}(0),\sigma(0)}=\e_{\sigma(0),\sigma}\e_{0,\sigma}$
by identity (\ref{actionidentity}).

Given $\sigma_1, \sigma_2 \in G$, consider first $\sigma_2\sigma_1
(\tilde{s}_{j,k})=\sigma_2(\frac{1}{d_{\sigma_1(0)}}\e_{\sigma_1(k),\sigma_1}\tilde{s}_{j,\sigma_1(k)})$

$=\sigma_2(\frac{1}{d_{\sigma_1(0)}}\e_{\sigma_1(k),\sigma_1}\tilde{s}_{\sigma_1(k),j})$$=
\frac{1}{d_{\sigma_2(0)}\lambda_{\sigma_1(0),\sigma_2(0)}}\e_{\sigma_1(k),\sigma_1}\e_{\sigma_2(j),
\sigma_2}\tilde{s}_{\sigma_1(k),\sigma_2(j)}$.

Then consider
\begin{eqnarray*} &&\sigma_1\sigma_2 (\tilde{s}_{j,k})=\sigma_1\sigma_2
(\tilde{s}_{k,j})=
\sigma_1(\frac{1}{d_{\sigma_2(0)}}\e_{\sigma_2(j),\sigma_2}\tilde{s}_{k,\sigma_2(j)})
=\\
&&\sigma_1(\frac{1}{d_{\sigma_2(0)}}\e_{\sigma_2(j),\sigma_2}\tilde{s}_{\sigma_2(j),k})
=\frac{1}{d_{\sigma_1(0)}\lambda_{\sigma_2(0),\sigma_1(0)}}\e_{\sigma_2(j),\sigma_2}\e_{\sigma_1(k),
\sigma_1}\tilde{s}_{\sigma_2(j),\sigma_1(k)}.
\end{eqnarray*}

Hence $\sigma_1\sigma_2=\sigma_2\sigma_1$ using
$d_i\lambda_{j,i}=d_j\lambda_{i,j}$, i.e. $G$ is abelian.

Suppose now that the rank $n=2r$ is even. Then $\det(\tilde S)^2 =
D^{2n}$ hence $\det(\tilde S)=\pm D^{2r}$. Since the determinant
is a polynomial in the entries of the matrix $\det(\sigma(\tilde
S))=\pm \sigma(D^2)^r$, with the same sign as $\det(\tilde S)$.
Hence $\det(\tilde S^{-1} \sigma(\tilde S))=d_{\sigma(0)}^{-n}$.
Since $\det(C_{\sigma})=d_{\sigma(0)}^{-n} (-1)^{\sigma}
\prod_{j=0}^{n-1} \e_{j,\sigma}$, we conclude $\prod_{j=0}^{n-1}
\e_{j,\sigma}=(-1)^{\sigma}$. For odd rank $n=2r+1$, $\det(\tilde
S)=\pm D^{2r+1}$, hence $D\in K$. Hence $\sigma(D)=\e_{\sigma}
D/d_{\sigma(0)}, \e_{\sigma}=\pm 1$ and one gets the formula
$\prod_{j=0}^{n-1} \e_{j,\sigma}=\e_{\sigma} (-1)^{\sigma}$.

\end{proof}

Note that the resulting equations (\ref{actionidentity}) for the
entries $\tilde s_{j,k}$ are unchanged if we replace $B_{\sigma}$
with $-B_{\sigma}$. We will use this to assume $\e_0=1$ below.

Next we will use the fact that the $\theta_i \in \U(1)$ to produce
a series of twist inequalities on the entries of $\tilde S$.

\begin{theorem}\label{twistinequlity}

Given a modular symbol $(\mN;S,T)$ and $S$ is a real matrix, then

\begin{enumerate}

\item $2 \max_i \tilde s_{i,j}^2 \le D\vert \tilde{s}_{jj}\vert +
D^2$ for any $j$. \label{jjtwistinequality}

\item If $j\neq k$, then $D \le \frac{1}{\vert \tilde
s_{j,k}\vert} \sum_{i=0}^{n-1} \vert \tilde s_{i,j}\tilde
s_{i,k}\vert.$\label{jktwistinequality}

\item $\sum_{j=0}^{n-1} \frac{\e_{\sigma(j)} \tilde
s_{j,\sigma(j)}}{\theta_j \theta_{\sigma(j)}} = D_-
\sum_{i:~\sigma(i)=i} \theta_i \e_{\sigma(i)}.$

\end{enumerate}

\end{theorem}

\begin{proof}

Rewrite the twist equation as $T \tilde S T \tilde S T = D_+
\tilde S$. Then taking the $(j,k)$ entry of this formula gives
$$\theta_j \theta_k \sum_{i=0}^{n-1} \theta_i \tilde s_{i,j}\tilde s_{i,k} =
D_+ \tilde s_{j,k}.$$ Since $\vert D_+\vert=D$ and $\vert
\theta_i\vert=1$, the largest of the $n+1$ numbers $\vert \tilde
s_{i,j}\tilde s_{i,k}\vert$, $0\le i\le n-1$, and $D \vert \tilde
s_{j,k}\vert$ must be at most the sum of the other $n$. If $j=k$,
then $\sum_i \tilde s_{i,j}^2 = D^2 > D\vert \tilde{s}_{jj}\vert
$. Hence this inequality is trivial unless the largest is one of
the first $n$ and we get
$$2 \max_i \tilde s_{i,j}^2 \le D \vert \tilde{s}_{jj}\vert+ \sum_{i=0}^{n-1} \tilde s_{i,j}^2.$$
If $j\ne k$ then $\sum_i \tilde s_{i,j}\tilde s_{i,k} = 0$ and the
nontrivial case is
$$D \le \frac{1}{\vert \tilde s_{j,k}\vert} \sum_{i=0}^{n-1}
\vert \tilde s_{i,j}\tilde s_{i,k}\vert.$$ We will refer to these
as the twist inequalities.

Suppose $\sigma\in G$ corresponds to signs $\e_i$ as above.  We
drop $\sigma$ for notational easiness.  Multiply the identity
above by $\e_{\sigma(j)}/(\theta_j \theta_{\sigma(j)})$, set
$k=\sigma(j)$, and sum over $j$. The result is
$$\sum_{j=0}^{n-1} \e_{\sigma(j)} \sum_{i=0}^{n-1} \theta_i \tilde s_{i,j}\tilde
s_{i,\sigma(j)} = D_+ \sum_{j=0}^{n-1} \frac{\e_{\sigma(j)} \tilde
s_{j,\sigma(j)}}{\theta_j \theta_{\sigma(j)}}.$$ Interchanging the
sums and using the fact that $\tilde
s_{i,\sigma(j)}=\e_{\sigma(j)} \e_{\sigma(i)} \tilde
s_{\sigma(i),j}$ gives
$$\sum_{i=0}^{n-1} \theta_i \e_{\sigma(i)} \sum_{j=0}^{n-1} \tilde s_{i,j}\tilde
s_{\sigma(i),j} = D_+ \sum_{j=0}^{n-1} \frac{\e_{\sigma(j)} \tilde
s_{j,\sigma(j)}}{\theta_j \theta_{\sigma(j)}}.$$ By orthogonality
of the rows of $\tilde S$, the innermost sum on the left is zero
if $i\ne \sigma(i)$ and $D^2=D_+D_-$ if $i=\sigma(i)$. Hence
$$\sum_{j=0}^{n-1} \frac{\e_{\sigma(j)} \tilde s_{j,\sigma(j)}}{\theta_j
\theta_{\sigma(j)}} = D_- \sum_{i:~\sigma(i)=i} \theta_i
\e_{\sigma(i)}.$$ If $\sigma$ is fixed point free, then
$\sum_{j=0}^{n-1} \frac{\e_{\sigma(j)} \tilde
s_{j,\sigma(j)}}{\theta_j \theta_{\sigma(j)}} =0$.

\end{proof}

\section{Rank=$2$ and $3$ modular $S$ matrices}\label{rank23}

In this section, we determine all possible modular $S$ matrices
for rank=$2$ and $3$ modular symbols.  The rank=$3$ case first appeared in \cite{CP}, but our
proof is new.

\begin{theorem}\label{Srank2}

The only possible rank=$2$ modular $\tS$ matrices of some modular
symbols are

\begin{enumerate}

\item
\[
 \begin{pmatrix}
   1& \e\\
   \e & -1
 \end{pmatrix},
 \]
 where $\e^2=1$;

\item

 \[
 \begin{pmatrix}
   1& \varphi \\
   \varphi & -1
 \end{pmatrix},
 \] where $\varphi^2=1+\varphi$.

 \end{enumerate}

 \end{theorem}

\begin{proof}

Since all labels are self-dual, $\tS$ is a symmetric real unitary
matrix of the form $
 \begin{pmatrix}
   1& d\\
   d& -1
 \end{pmatrix}
 $.  The fusion matrix $N_1$ is of the form $
 \begin{pmatrix}
   0& 1\\
   1 & m
 \end{pmatrix},$  so we have $d^2=1+md$.  Simplifying
 $D_+D_{-}=D^2$ leads to $\theta+\theta^{-1}=1-d^2=-m\cdot d$.
 Since $\theta \in \U(1)$, so $|md|\leq 2$.
 If $d>0$, then $d=\frac{m+\sqrt{m^2+4}}{2}$, hence $m=0,1$.  If $d<0$,
 then $d=\frac{m-\sqrt{m^2+4}}{2}$, hence $[\Q(\theta+\theta^{-1}):\Q]\leq 2$.  It follows that
 $\theta=e^{\frac{p \pi i}{q}}$ for some $(p,q)=1$, and $q$ is one of $\{1,2,3,4,5,6\}$.  Direct computation
 shows there are no integral solutions $p,q$ for $2\cos(\frac{p\pi}{q})=-m\cdot \frac{m-\sqrt{m^2+4}}{2}$
  except for $q=2,5$ and $m=0,1$.

 \end{proof}

 \begin{theorem}\label{Srank3}

 Then
 the only possible rank=$3$ modular $\tS$ matrices of some modular symbols up to permutations are

 \begin{enumerate}

 \item

  \[
 \begin{pmatrix}
   1& \e& \e\\
   \e& \omega &\omega^2 \\
   \e& \omega^2 &\omega
 \end{pmatrix},
 \]

 where $\e^2=1$, and $\omega^3=1, \omega\neq 1$.

 \item

 \[
 \begin{pmatrix}
   1& d & 1 \\
   d & 0 & -d \\
   1 & -d  & 1
 \end{pmatrix},
 \] where $d^2=2$.

 \item

  \[
 \begin{pmatrix}
   1& d_1 & d_2 \\
   d_1 & -d_2 & 1 \\
   d_2 & 1  & -d_1
 \end{pmatrix},
 \]
 where $d_1$ is a real root of $x^3-2 x^2 - x +1$ and $d_2=d_1/(d_1-1)$ which is a root of $x^3-x^2-2x+1$.
 The largest $d_1=\frac{2\cos(\pi/7)}{2\cos(\pi/7)-1}
=2.246979604\dots,$
and $d_2=2\cos(\pi/7)=1.801937736\dots$.
 \end{enumerate}

\end{theorem}

\begin{proof}

The case of non-self dual case is given in Appendix
\ref{nonselfdual}. Hence we assume all fusion rules are self-dual,
so $\tS$ is a real, symmetric, unitary matrix up to the scalar $D$.
  It follows that the fusion matrices $N_i$'s are commutative, symmetric, integral matrices.
One approach to proving the theorem is to analyze case by case for
the Galois groups of fusion rules $G\cong 1, \Z_2,\Z_3$.  This
strategy will be fully exploited in the rank=$4$ case in the next
section. Instead we will argue directly from the $\tS$-matrix in
this section.

The fusion matrices $N_1,N_2$ are symmetric, and $N_1N_2=N_2N_1$.
Therefore, they can be written as \[ N_1=
 \begin{pmatrix}
   0& 1 & 0 \\
   1 & m & k \\
   0 & k & l
 \end{pmatrix}
 \] and

 \[ N_2=
 \begin{pmatrix}
   0& 0 & 1 \\
   0 & k & l \\
   1 & l  & n
 \end{pmatrix}
 \] such that $$1+ml+kn=k^2+l^2.$$  There characteristic polynomials
 are
$$p_1(x)=x^3 - (\ell+m) x^2 + (m\ell -k^2-1) x +\ell=0$$
and
$$p_2(x)=x^3 - (k+n) x^2 + (nk - \ell^2 -1) x + k =0,$$
respectively.

Next we turn to the $\tS$ matrix, which is of the following form:
\[ \tS=
 \begin{pmatrix}
   1& d_1 & d_2\\
   d_1 & \ts_{11}& \ts_{12} \\
   d_2 & \ts_{12} & \ts_{22}
 \end{pmatrix}.
 \]
Orthogonality of the columns of the $\tS$ matrix translates into
the equations
$$d_1 + d_1 \tilde{s}_{11} + d_2 \tilde{s}_{12}=0,$$
$$d_2 + d_1 \ts_{12} + d_2 \ts_{22}=0,$$
$$d_1 d_2 +\ts_{12} (\ts_{11}+\ts_{22})=0.$$
The first two equations give $\ts_{11}=-1-d_2 \ts_{12}/d_1$ and
$\ts_{22}=-1-d_1 \ts_{12}/d_2$. Plugging these into the third
equation gives
$$(d_1^2 + d_2^2) \ts_{12}^2 + 2 d_1 d_2 \ts_{12} - d_1^2 d_2^2=0$$
hence
$$\ts_{12} =\frac{d_1 d_2}{1\pm D}.$$
Thus
$$\ts_{11} = -1 - \frac{d_2^2}{1\pm D},$$
and
$$\ts_{22} = -1 -\frac{d_1^2}{1\pm D}.$$
Thus the eigenvalues of $N_1$ are
$$d_1, b=\frac{\ts_{12}}{d_2} = \frac{d_1}{1\pm D},\; {\rm and }\;
c=\frac{\ts_{11}}{d_1}=-\frac{1}{d_1} -\frac{d_2^2}{d_1 (1\pm
D)}$$ and the eigenvalues of $N_2$ are
$$d_2, e=\frac{\ts_{12}}{d_1} = \frac{d_2}{1\pm D},\; {\rm and }\;
f=-\frac{1}{d_2} -\frac{d_1^2}{d_2 (1\pm D)}.$$ We compute
$$d_1 b + d_2 f = d_1 c + d_2 e = b c + e f= -1.$$
Since $d_1 b c =-\ell$ and $d_2 e f=-k$, these are equivalent to
$$\frac{\ell}{c}+\frac{k}{e}=\frac{\ell}{b}+\frac{k}{f}
=\frac{\ell}{d_1}+\frac{k}{d_2} = 1.$$  Also note that
$$d_1 e = d_2 b.$$

 Let's deal with the case where $\ell=0$ first. Then
we have $k^2=kn+1$. Hence $k=1$ and $n=0$. Thus the eigenvalues of
$N_2$ are 1, 1, and -1 and the eigenvalues of $N_1$ are $(m+{\sqrt
{m^2+8}})/2$, 0, and $(m-{\sqrt {m^2+8}})/2$. Since $N_1$ has
eigenvalues $d_1,b,c$, and $d_1\neq 0$, hence $c=0$ which implies
$m=0$. This gives $(k,\ell,m,n)=(1,0,0,0)$ and
$$\tS=\left(\begin{matrix} 1 & d & 1\cr
d & 0 & -d\cr 1 & -d & 1\end{matrix}\right),$$ where $d^2=2$.  The
case $k=0$ gives essentially the same solution, so we will
henceforth assume $\ell$ and $k$ are positive. Since $p_1(\ell) =
-k^2 \ell \le 0$ and $p_1(0)=\ell \ge 0$, we see that the largest
root of $p_1$ is $> \ell$, one of the remaining roots is in
$(0,\ell)$ and the other root is negative. Similarly the largest
root of $p_2$ is $>k$ and the other roots are in $(0,k)$ and
$(-\infty,0)$.

\vskip .10in

\noindent {\bf Case 1.} The polynomial $p_1(x)$ is reducible.

Since $d_1 > \ell$, $d_1$ cannot be an integer. Thus $p_1$ must
split into a linear and an irreducible quadratic. Thus ${\Q}[d_1,D]$ is a quadratic extension of ${\Q}$.
Hence ${\Q}[d_1,d_2,D]$ has degree 2 or 4 over ${\Q}$. Thus $p_2$ is
also reducible and also splits as into a linear and an irreducible
quadratic. Since the $\ell/b+k/f=\ell/c+k/e=1$, the integral roots
must be either $b$ and $f$ or $c$ and $e$. Without loss, we may
assume the integer roots are $b$ and $f$. Let
$$d_2=\alpha+\beta{\sqrt s}\; {\rm and }\; e=\alpha-\beta{\sqrt s}$$
for rational (in fact integer or half-integer) $\alpha$ and
$\beta$ and integer $s$. Then since $d_1 e = d_2 b$ and $c$ is the
conjugate of $d_1$ we have
$$d_1=b\frac{\alpha+\beta{\sqrt s}}{\alpha-\beta{\sqrt s}}\; {\rm and }\;
c=b\frac{\alpha-\beta{\sqrt s}}{\alpha+\beta{\sqrt s}}.$$ Hence
$\ell = - d_1 b c = - b^3$. Since $f=-k/(d_2 e) =
-k/(\alpha^2-\beta^2 s)$ and $-b^2=\ell/b=1-k/f=\alpha^2+1-\beta^2
s$. Therefore solving $1=k/d_2+\ell/d_1$ for $k$ gives
\begin{equation*}
\begin{split}
k & =d_2 -\ell \frac{d_2}{d_1} = \alpha+\beta {\sqrt s} + b^3
\frac{e}{b}
= \alpha+\beta {\sqrt s} - (\alpha^2+1-\beta^2 s) (\alpha-\beta {\sqrt s})\\
& =-\alpha (\alpha^2-\beta^2 s) + \beta [\alpha^2+2-\beta^2
s]{\sqrt s}.
\end{split}
\end{equation*}
Since $k$ is an integer, this forces $\alpha^2-\beta^2 s=-2$,
hence $b^2 = 1$. Since $\ell >0$, this means $\ell=1$ and $b=-1$.
Also from the equations above we get $k=2\alpha$,
$d_2=\alpha+{\sqrt {\alpha^2 + 2}}$, $e=\alpha-{\sqrt {\alpha^2 +
2}}$, $f=\alpha$, $d_1=\alpha^2 + 1 +\alpha {\sqrt
{\alpha^2+2}}=d_2^2/2$, $c=\alpha^2 + 1 -\alpha {\sqrt
{\alpha^2+2}}$, and $D=\alpha^2 +2 +\alpha {\sqrt {\alpha^2 +2}}$.
Thus
$$p_1(x)=x^3 - (2 \alpha^2+1) x^2 - (2 \alpha^2 +1)x +1$$
and
$$p_2(x)=x^3 - 3\alpha x^2 + (2\alpha^2 -2) x +2\alpha.$$
Thus $(k,\ell,m,n)=(2\alpha, 1, 2\alpha^2, \alpha)$ and
$$\tS=\left( \begin{matrix} 1 &
\alpha^2+1+\alpha {\sqrt {\alpha^2 +2}} & \alpha+{\sqrt
{\alpha^2+2}} \cr \alpha^2+1+\alpha {\sqrt {\alpha^2 +2}} & 1 &
-\alpha-{\sqrt {\alpha^2+2}} \cr \alpha+ {\sqrt {\alpha^2 +2}} &
-\alpha-{\sqrt {\alpha^2+2}} & \alpha^2+\alpha {\sqrt {\alpha^2
+2}}\end{matrix}\right).$$ Note that $n=\alpha$ must be a
non-negative integer. Setting $\alpha=0$ gives the example found
above again. Thus we may assume $\alpha\ge 1$. Since
$d_1=d_2^2/2$, the equation for the $\theta$'s is
$$\vert 1+\theta_2 d_2^2 + (1/4) \theta_1 d_2^4\vert = D = 1 + (1/2) d_2^2.$$
If a solution did exist, then we would have $1+ (1/2) d_2^2 \ge
(1/4) d_2^4 - d_2^2 -1$, hence $17\ge (d_2^2-3)^2$ or $d_2 \le
{\sqrt {3+ \sqrt {17}}}=2.66891\dots$. Since $d_2 \ge 1+ {\sqrt
{3}} =2.73205\dots$, this cannot occur. Thus these $\tS$ matrices,
for positive $\alpha$, do not give a modular symbol.

\vskip .10in

\noindent {\bf Case 2.} The polynomial $p_1(x)$ is an irreducible
cubic.

By Case 1, we see that $p_2(x)$ is also an irreducible cubic. Then
there must be a Galois symmetry $\sigma$ with $\sigma(d_1)=b$,
$\sigma(b)=c$ and $\sigma(c)=d_1$. Hence $\sigma(d_2)=f$,
$\sigma(f)=e$, and $\sigma(e)=d_2$ since these roots of $p_2$ pair
with the corresponding roots of $p_1$. Applying $\sigma$ to the
identity $d_1 e = d_2 b$, gives $d_2 b = f c$. Thus we must have
$f c = d_1 d_2/(1\pm D)$. Since
$$f c = \frac{d_1 d_2}{(1\pm D)^2} + \frac{d_1^2 + d_2^2}{d_1 d_2 (1\pm D)}
+\frac{1}{d_1 d_2} =  \frac{d_1 d_2}{(1\pm D)^2} + \frac{D^2 \pm
D}{d_1 d_2 (1\pm D)},$$ we compute
$$1\pm D = \frac{(1\pm D)^2}{d_1 d_2} f c
= 1 \pm \frac{D (1\pm D)^2}{d_1^2 d_2^2}$$ and hence
$$\left(\frac{1\pm D}{d_1d_2}\right)^2 = 1.$$
Since $D>1$, we get that
$$\ts_{12}=\frac{d_1 d_2}{1\pm D} =\pm 1$$
and hence $b=\pm 1/d_2$ and $e=\pm 1/d_1$. Thus $d_1$ and $d_2$
are units in the ring of algebraic integers. Hence $k=\ell=1$ and
hence $m+n=1$. Without loss we may assume $m=1$ and $n=0$. Then
$p_1(x)=x^3-2 x^2 - x +1$ and $p_2(x)=x^3 p_1(1/x)=x^3 - x^2 - 2 x
+1$. Then one computes
$$d_1 = \frac{2\cos(\pi/7)}{2\cos(\pi/7)-1}
=2.246979604\dots,$$ $b=1-1/d_1$, $c=-1/(d_1-1)$,
$d_2=d_1/(d_1-1)=1.801937736\dots$, $e=1/d_1$, and $f=1-d_1$ and
$$\tS=\left(\begin{matrix} 1 & d_1 & d_2\cr d_1 & -d_2 & 1\cr d_2 & 1
& -d_1 \end{matrix}\right).$$

\end{proof}

\vskip .10in

\section{Rank=4 modular $S$ matrices}\label{rankIV}

First we introduce the following notation.  For an integer $m$,
define
$$\phi_m = \frac{m + {\sqrt {m^2+4}}}{2},$$
that is, $\phi_m$ is the unique positive root of  $x^2-mx-1=0$.
Note that any algebraic number $\phi$ whose only conjugate is
$-1/\phi$ must be $\phi_m$ for some integer $m$. Also note the
only rational $\phi_m$ is $\phi_0=1$.

\begin{theorem}\label{Srank4}

The only possible rank=$4$ modular $\tS$ matrices of
\emph{unitary} modular symbols up to permutations are

\begin{enumerate}

\item

  \[
 \begin{pmatrix}
   1& 1& 1 & 1 \\
   1 & 1& -1 &-1\\
   1 & -1&\omega &\bar{\omega}\\
   1&-1&\bar{\omega} & \omega
 \end{pmatrix},
 \]
where $\omega=\pm i$;

\item

  \[
 \begin{pmatrix}
   1& 1& 1 &1 \\
   1 & 1& -1&-1  \\
   1 & -1 &1 &-1 \\
   1 &-1 & -1& 1
 \end{pmatrix};
 \]

 \item

  \[
 \begin{pmatrix}
   1& 1& 1 & 1 \\
   1 & -1& 1 &-1 \\
   1 & 1 &-1 &-1 \\
   1 &-1 &-1 & 1
 \end{pmatrix};
 \]

 \item

  \[
 \begin{pmatrix}
   1& \varphi& 1 & \varphi \\
   \varphi & -1& \varphi &-1  \\
   1& \varphi &-1 &-\varphi \\
   \varphi &-1 &-\varphi & 1
 \end{pmatrix},
 \] where $\varphi=\frac{1+\sqrt{5}}{2}$ is the golden ratio;

 \item

 \[
 \begin{pmatrix}
   1& \varphi & \varphi &\varphi^2 \\
   \varphi & -1& \varphi^2 &-\varphi  \\
   \varphi & \varphi^2 &-1 &-\varphi \\
   \varphi^2 &-\varphi &-\varphi & 1
 \end{pmatrix};
 \]

 \item

  \[
 \begin{pmatrix}
   1& d^2-1& d+1 &d \\
   d^2-1 & 0& -d^2+1 & d^2-1  \\
   d+1& -d^2+1 & d &-1 \\
   d &d^2-1 &-1 & -d-1
 \end{pmatrix},
 \] where $d$ is the largest real root of $x^3-3x-1$.

\end{enumerate}

\end{theorem}

\begin{proof}

The non-self dual case is treated in Appendix \ref{nonselfdual},
so we will assume that $\tS$ is real in the following.  Since the
fusion coefficients $n_{i,j}^k$ are totally symmetric in $i,j$ and
$k$ for self-dual categories, we will instead write $n_{i,j,k}$ in
what follows.  For notational easiness, when the Galois group element $\sigma$ is clear from
the context, we simply write $\epsilon_{i,\sigma}$ as $\epsilon_i$.  Identities for $\epsilon_i$ and
$\tilde{s}_{jk}$ that are not referenced are all from Theorem \ref{Galoissymmetry}.  All the
twist inequalities are from Theorem \ref{twistinequlity}.

\vspace{.5cm}

 \noindent {\bf Case 1.} $G$ contains a 4-cycle.

\vspace{.5cm}

By symmetry we may assume $\sigma=(0\;\;1\;\;2\;\; 3)\in G$. The
conditions $\tilde s_{j,k}=\e_k \e_{\sigma(j)} \tilde
s_{\sigma(j),\sigma^{-1}(k)}$ and $\e_1 \e_2 \e_3=-1$ give
$$\tilde S=\left(\begin{matrix} 1 & d_1 & d_2 & d_3\cr d_1 & \e_1\e_2 d_2 &
-\e_2 d_3 & \e_1\cr d_2 & -\e_2 d_3 & -1 & \e_2 d_1\cr d_3 & \e_1
& \e_2 d_1 & -\e_1\e_2 d_2\end{matrix}\right).$$ By symmetry under
interchanging $N_1$ and $N_3$, we may assume $\e_2=+1$. Note that
$\sigma^2(d_2)=\lambda_{2,2}=-1/d_2\ne d_2$. Hence the
characteristic polynomial $p_2$ of $N_2$ is irreducible. Since
$\sigma^2(d_1)=-d_3/d_2<0$. Hence $\sigma^2(d_1)\ne d_1$. Thus
$p_1$ is irreducible. Since $\e_1/d_3$ is a root of $p_1$, it
follows that $p_3$ is also irreducible.

We see that $\lambda_{1,1}=\e_1 d_2/d_1$,
$\lambda_{1,2}=-d_3/d_2$, and $\lambda_{1,3}=\e_1/d_3$. In
particular $d_1\lambda_{1,1}\lambda_{1,2}\lambda_{1,3}=-1$.
Orthogonality of the rows of $\tilde S$ is equivalent to
$$d_1+\e_1 d_1 d_2 - d_2 d_3 +\e_1 d_3=0,\;\;{\rm or}$$
$$\frac{1}{d_1}+\lambda_{1,3}=\lambda_{1,1}+\frac{1}{\lambda_{1,2}}.$$
Write $p_1(x)=x^4 -c_1 x^3 + c_2 x^2 + c_3 x-1$. Then $p_4(x)=x^4
-\e_1 c_3 x^3 - c_2 x^2 + \e_1 c_1 x -1.$ Note that $c_1={\rm
Trace}(N_1)\ge 0$ and $\e_1 c_3={\rm Trace}(N_3)\ge 0$.
Multiplying together the orthogonality condition above and five of
its formal conjugates gives
$$128 + (c_3^2-c_1^2)^2 -16 c_1 c_3 + 12 (c_3^2-c_1^2)c_2 + 32 c_2^2=0.$$
This equation forces $c_1$ and $c_3$ to be even. Let
$\Delta=c_1-c_3$ and $\Sigma=c_1+c_3$ (hence $\Sigma$ and $\Delta$
are even and congruent mod 4). Then solving the quadratic equation
above for $c_2$ we see that we must have
$(\Sigma^2-32)(\Delta^2+32)$ to be a square and
$c_2=\frac{3\Delta\Sigma \pm {\sqrt {(\Sigma^2-32)(\Delta^2+32)}}}{16}.$
It follows that $\vert \Sigma\vert \ge 6$. If $\Sigma$ and
$\Delta$ are multiples of 4, then we see they are multiples of 8
and either sign gives an integral $c_2$. If $\Sigma$ and $\Delta$
are both 2 mod 4, then there is a unique choice of the sign for
which $c_2$ is integral.

The Galois group of $p_1$ must be $\Z/4\Z$, otherwise it would
contain the $(0\;\; 1)(2\;\; 3)$. Applying this to the
orthogonality identity above gives
$1/\lambda_{1,3}+d_1=\lambda_{1,2}+1/\lambda_{1,1}$. Multiplying
this by the original identity gives
$\frac{1}{d_1\lambda_{1,3}}+d_1\lambda_{1,3}=
\frac{1}{\lambda_{1,1}\lambda_{1,2}}+\lambda_{1,1}\lambda_{1,2}.$
Hence $d_1\lambda_{1,3}=(\lambda_{1,1}\lambda_{1,2})^{\pm 1}$,
either of which contradicts the product of all four roots being
$-1$. In particular, $p_1$ cannot have complex roots, since
complex conjugation would give a transposition in the Galois
group. Applying $\sigma$ to the orthogonality identity gives
$d_1+1/\lambda_{1,1}=\lambda_{1,2}+1/\lambda_{1,3}$.

We know from the preliminary discussion that all three of the
resulting $N_i$ matrices will be rational. Define
$P=\frac{16c_2-3\Delta\Sigma}{\Delta^2+32}=
\pm {\sqrt {\frac{\Sigma^2-32}{\Delta^2+32}}},$ then we compute
\begin{equation*}\begin{split}
n_{1,1,1} & =\frac{5c_1-3c_3}{8} - \frac{c_1-c_3}{8} P,\\
n_{1,1,2} & = \e_1 (P-1),\\
n_{1,1,3} & = \e_1 \left(\frac{c_1+c_3}{8}-\frac{c_1-c_3}{8}P\right),\\
n_{1,2,2} & = \frac{c_1+c_3}{4} + \frac{c_1-c_3}{4}P,\\
n_{1,2,3} & = P,\\
n_{1,3,3} & = \frac{c_1+c_3}{8}-\frac{c_1-c_3}{8}P,\\
n_{2,2,2} & = \e_1\left(2c_2-\frac{c_3^2+c_1^2}{4}-2P\right),\\
n_{2,2,3} & = \e_1\left(\frac{c_1+c_3}{4} + \frac{c_1-c_3}{4}P\right),\\
n_{2,3,3} & = \e_1 (P+1),\;\; {\rm and}\\
n_{3,3,3} & = \e_1\left(
\frac{5c_3-3c_1}{8}-\frac{c_1-c_3}{8}P\right).
\end{split}\end{equation*}
Recall that the $n_{i,j,k}$ must be nonnegative integers. This
restricts the $c_i$. First looking at $n_{1,2,3}$, we see that $P$
must be a positive integer. Hence $c_2$ must be given by the upper
sign. This condition in fact guarantees integrality of all the
$n_{i,j,k}$. (The additional factors of 2 in the denominator
cancel out if $c_2$ is integral and can be ignored.) Integrality
of $P$ severely restricts $\Delta$, since it requires all odd
prime factors of $\Delta^2+32$ to be congruent to 1 mod 8 (since
$2$ and $-2$ are both squares mod any such prime). In particular
either $\Delta=0$ or $\vert \Delta\vert \ge 6$. Since $P\ne 0$, we
see that $\Sigma^2\ge \Delta^2+64$, hence $\vert\Sigma\vert >
\vert \Delta\vert$. Thus $c_3$ must be positive and $\Sigma>0$.
Since we saw above $\e_1 c_3\ge 0$, we see that $\e_1=+1$. Thus
rewriting the orthogonality relation gives
$d_1/d_3=(d_2-1)/(d_2+1)$.

The twist inequality coming from the $(0,3)$ entry reads
$$D\le \left(1+\frac{d_1}{d_3}\right)(1+d_2).$$
Plugging in the preceding identity, simplifies this to $D\le 2
d_2$. Rearranging gives $3 d_2^2 \ge d_1^2 + d_3^2 + 1 > d_1^2 +
d_3^2$ and plugging in the identity $d_1=d_3 (d_2-1)/(d_2+1)$
yields
$$\left(\frac{d_3}{d_2+1}\right)^2 < \frac{3d_2^2}{2(d_2^2+1)}<\frac{3}{2}.$$
To see why this is helpful, expand the equations ${\rm
Trace}(N_i)=c_i$ for $i=1,3$ and use the identity above to
eliminate $d_1$. The result is
$$c_1=\frac{d_2^2-2d_2-1}{d_2(d_2+1)} d_3 +
\frac{d_2^2+2d_2-1}{d_2-1}\cdot \frac{1}{d_3},\;\; {\rm and}$$
$$c_3=\frac{d_2^2+2d_2-1}{d_2(d_2+1)} d_3 -
\frac{d_2^2-2d_2-1}{d_2-1}\cdot \frac{1}{d_3}.$$ Subtracting these
gives
$$\Delta=c_1-c_3=-4 \frac{d_3}{d_2+1} + 2 \frac{d_2+1}{d_2d_3}.$$
Hence $\Delta > -4 \frac{d_3}{d_2+1} > -4 {\sqrt {3/2}}>-4.9$.
However, we saw above that either $\Delta=0$ or $\vert \Delta\vert
\ge 6$. It follows that $\Delta\ge 0$. Since $\Sigma\ge 6$ and
$c_2\ge 3 \Delta\Sigma/16$, it follows that $c_2>\Delta=c_1-c_3$.
Thus $p_3(1)=c_1-c_2-c_3<0$. Thus $d_3>1$. Hence we see $\Delta <
2(d_2+1)/(d_2d_3)<4$. It follows that $\Delta=0$, i.e., $c_1=c_3$.

Since $\Delta=0$, $\Sigma=2c_1$ is a multiple of 8 and
$$\left(\frac{P}{32}\right)^2-2\left(\frac{c_1}{4}\right)^2 = -1.$$
In this case the characteristic polynomials become
$$p_1(x)=x^4 - c_1 x^3 + 2 {\sqrt {2(c_1/4)^2-1}} x^2 + c_1 x -1,$$
$$p_2(x)=x^4 - 4 {\sqrt {2(c_1/4)^2-1}} x^3 -6 x^2 +4 {\sqrt {2(c_1/4)^2-1}} x
+ 1,\;\; {\rm and}$$
$$p_3(x)=x^4 - c_1 x^3 - 2 {\sqrt {2(c_1/4)^2-1}} x^2 + c_1 x -1.$$
In particular $p_1(x)>0$ for $x\ge c_1$. Hence $d_1<c_1$.
We have
$$p_2(x) = (x^2 - t_1 x -1)(x_2 - t_2 x - 1),$$
where $t_1>0>t_2$ are the two roots of $t^2 - 4 {\sqrt
{2(c_1/4)^2-1}} t - 4 =0$. Since the larger root of $x^2-tx-1$ is
an increasing function of $t$, $d_2$ must correspond to $t_1$.
Hence
$$d_2 = t_1 + \frac{1}{d_2} > t_1 = 4 {\sqrt {2(c_1/4)^2-1}} + \frac{4}{t_1}
> 4 {\sqrt {2(c_1/4)^2-1}}.$$
In particular, $d_2>4$ since the square root above is integral.
Finally the twist inequality coming from the $(0,2)$ entry reads
$$D\le 2 \left( 1+ \frac{d_1d_3}{d_2}\right).$$
Squaring and using the identity $d_3=d_1 (d_2+1)/(d_2-1)$ to
eliminate $d_3$ gives the inequality
$$4(d_2+1)^2 d_1^4 -2 d_2 (d_2^3-4d_2^2 + d_2 + 4) d_1^2
- d_2^2 (d_2-1)^2 (d_2^2-3)\ge 0.$$ Dividing through by
$4(d_2+1)^2 d_1^2$ and rearranging gives
$$d_1^2 \ge \frac{d_2}{2(d_2+1)^2}(d_2^3-4d_2^2 + d_2 + 4) +
\frac{d_2^2(d_2-1)^2(d_2^2-3)}{4(d_2+1)^2d_1^2}.$$ The right hand
side of this inequality is an increasing function of $d_2$ for
$d_2>4$ and a decreasing function of $d_1$, hence we may replace
$d_1$ by its upper bound $c_1$ and $d_2$ by the lower bound above.
The result is
$$4560 - 2138 c_1^2 + 264 c_1^4 - 8 c_1^6 +
(608 - 276 c_1^2 + 32 c_1^4) {\sqrt {2 c_1^2-16}}\ge 0.$$ Since
the coefficient of the square root is nonnegative it follows that
$$4560 - 2138 c_1^2 + 264 c_1^4 - 8 c_1^6 +
(608 - 276 c_1^2 + 32 c_1^4) c_1 {\sqrt {2}}\ge 0.$$ This
polynomial in $c_1$ is negative for $c_1>6 {\sqrt 2}$, hence
$c_1\le 8$. The only multiple of 4 in the range $3\le c_1\le 8$
for which the Pell equation above is satisfiable is $c_1=4$. This
gives $c_1=c_3=4$, $c_2=2$, $P=32$. However plugging in shows that
the twist inequality $D\le 2(1+ d_1 d_3/d_2)$ does not actually
hold in this case. Thus there are no solutions in this case.

\vspace{.5cm}

\noindent {\bf Case 2.} $G$ is the Klein 4-group.

\vspace{.5cm}

Let $\sigma_1$, $\sigma_2$, and $\sigma_3$ be the elements of $G$
which correspond to $(0\;\; 1)(2\;\; 3)$, $(0\;\; 2)(1\;\; 3)$,
and $(0\;\; 3)(1\;\; 2)$, respectively. Let $B_{\sigma_1}$
correspond to signs $\e_i$ with $\e_1\e_2\e_3=1$ and
$B_{\sigma_2}$ correspond to signs $\delta_i$ with
$\delta_1\delta_2\delta_3=1$. Then using the usual identities
gives
$$\tilde S=\left(\begin{matrix} 1 & d_1 & d_2 & d_3\cr d_1 & \e_1 & \e_2
d_3 & \e_1\e_2 d_2\cr d_2 & \e_2 d_3 & \tilde s_{2,2} & \tilde
s_{2,3}\cr d_3 &  \e_1\e_2 d_2 & \tilde s_{2,3} & \e_1 \tilde
s_{2,2}\end{matrix} \right) = \left(\begin{matrix} 1 & d_1 & d_2 &
d_3\cr d_1 & \tilde s_{1,1} & \delta_1 d_3 & \tilde s_{1,3}\cr d_2
& \delta_1 d_3 & \delta_2 & \delta_1\delta_2 d_1\cr d_3 & \tilde
s_{1,3} & \delta_1\delta_2 d_1 & \delta_2 \tilde
s_{1,1}\end{matrix}\right).$$ Comparing these we see
$\e_2=\delta_1$, $\tilde s_{1,1}=\e_1$, and $\tilde
s_{2,2}=\delta_2$, hence
$$\tilde S=\left(\begin{matrix} 1 & d_1 & d_2 & d_3\cr d_1 & \e_1 & \e_2
d_3 & \e_1\e_2 d_2\cr d_2 & \e_2 d_3 & \delta_2 & \e_2\delta_2
d_1\cr d_3 & \e_1\e_2 d_2 & \e_2\delta_2 d_1 & \e_1
\delta_2\end{matrix}\right).$$ Orthogonality of the rows of
$\tilde S$ gives the three conditions
$$(1+\e_1)(d_1+\e_2 d_2 d_3)=
(1+\delta_2)(d_2+\e_2 d_1 d_3)=(1+\e_1\delta_2)(d_3+\e_1\e_2
d_1d_2)=0.$$ Suppose $\e_1=+1$, then we see $d_1=-\e_2 d_2d_3$,
hence $\e_2=-1$ and the remaining orthogonality relations become
$d_2(1+\delta_2)(1-d_3)=d_3(1+\delta_2)(1-d_2)$. We cannot have
$d_2=d_3=1$, since this would make $d_1=1$, hence $\delta_2=-1$.
This gives
$$\tilde S=\left(\begin{matrix} 1 & d_2 d_3 & d_2 & d_3\cr d_2 d_3 & 1 &
-d_3 & -d_2\cr d_2 & -d_3 & -1 & d_2 d_3\cr d_3 & -d_2 & d_2 d_3 &
-1\end{matrix}\right).$$ The eigenvalues of $N_2$ are $d_2$ and
$-1/d_2$ each with multiplicity 2.
Hence $d_2=\phi_m$ for some integer $m$. The eigenvalues of $N_3$
are $d_3$ and $-1/d_3$ each with multiplicity 2, hence
$d_3=\phi_n$ for some integer $n$. So
$$\tilde S=\left(\begin{matrix} 1 & \phi_m \phi_n & \phi_m & \phi_n\cr
\phi_m \phi_n & 1 & -\phi_n & -\phi_m\cr \phi_m & -\phi_n & -1 &
\phi_m \phi_n\cr \phi_n & -\phi_m & \phi_m \phi_n &
-1\end{matrix}\right).$$ The resulting $N_i$ matrices are
necessarily rational, but in this case they are all integral,
namely,
$$N_1=N_2N_3=N_3N_2=\left(\begin{matrix} 0 & 1 & 0 & 0\cr 1 & mn & m & n\cr
0 & m & 0 & 1\cr 0 & n & 1 & 0\end{matrix}\right),$$
$$N_2=\left(\begin{matrix} 0 & 0 & 1 & 0\cr 0 & m & 0 & 1\cr 1 & 0 & m &
0\cr 0 & 1 & 0 & 0\end{matrix}\right),\;\; {\rm and}\;\;
N_3=\left(\begin{matrix} 0 & 0 & 0 & 1\cr 0 & n & 1 & 0\cr 0 & 1 &
0 & 0\cr 1 & 0 & 0 & n\end{matrix}\right).$$ Note that
nonnegativity of the entries forces $m,n\ge 0$ and hence
$\phi_m,\phi_n\ge 1$. The strongest twist inequalities are the
$(0,1)$ and $(2,3)$ cases which give $D\le 4$ or
$(\phi_m^2+1)(\phi_n^2+1)\le 16$. This gives, up to symmetry, the
solutions $(m,n)=(0,0)$, $(0,1)$, $(0,2)$, or $(1,1)$. These are
all excluded since the resulting Galois group $G$ is at most
$\Z/2\Z$. (These examples will return when we look at smaller
Galois groups.)

\vspace{.5cm}

\noindent {\bf Case 3.} $G$ contains a 3-cycle.

\vspace{.5cm}

Since we can exclude cases 1 and 2 above, the image of $G$ in
$S_4$ cannot be transitive. It follows that $G$ must fix the point
$j$ not on the 3-cycle. Thus $\lambda_{i,j}$ is rational (hence
integral) for every $i$. Up to symmetry there are two cases for
the 3-cycle. We could have $\sigma=(1\;\; 2\;\; 3)$ or
$\sigma=(0\;\; 1\;\; 2)$. If $\sigma=(1\;\; 2\;\; 3)$, then the
$d_i$ are integral. The identities $\tilde
s_{j,k}=\e_{\sigma(j)}\e_k \tilde s_{\sigma(j),\sigma^{-1}(k)}$
and $\e_1\e_2\e_3=1$ give $\e_i=1$ for all $i$ (since $d_i=\tilde
s_{0,i}=\e_i \tilde s_{0,i+1}=\e_i d_{i+1}$ for $1\le i\le 2$) and
$$\tilde S=\left(\begin{matrix} 1 & d_1 & d_1 & d_1\cr d_1 & \tilde s_{1,1}
& \tilde s_{3,3} & \tilde s_{2,2}\cr d_1 & \tilde s_{3,3} & \tilde
s_{2,2} & \tilde s_{1,1}\cr d_1 & \tilde s_{2,2} & \tilde s_{1,1}
& \tilde s_{3,3}
\end{matrix}\right).$$
Orthogonality of the columns of $\tilde S$ gives $\tilde
s_{1,1}+\tilde s_{2,2}+\tilde s_{3,3}=-1$ and $\tilde
s_{1,1}\tilde s_{2,2}+\tilde s_{2,2}\tilde s_{3,3}+\tilde
s_{3,3}\tilde s_{1,1}=-d_1^2$. The first of these gives
$-1/d_1=\lambda_{1,1}+\lambda_{1,2}+\lambda_{1,3}$ from which we
see $-1/d_1$ is an algebraic integer. Hence $d_1=1$ and
$\lambda_{1,i}<1$. The second equation gives
$\lambda_{1,1}\lambda_{1,2}+
\lambda_{1,2}\lambda_{1,3}+\lambda_{1,3}\lambda_{1,1}=-1$. Hence
$\lambda_{1,1}$, $\lambda_{1,2}$, and $\lambda_{1,3}$ are the
three roots of $g(x)=x^3+x^2-x+n$ for some integer $n$. This cubic
must be irreducible and have three real roots all less than 1.
Irreducibility excludes $n=0$ and $n=-1$. For the roots of $g$ to
be less than 1, we must have $g(1)>0$ or $n+1>0$. Hence $n\ge 1$.
However, this results in complex roots. Thus this case gives no
solutions.

Thus we must have $\sigma=(0\;\; 1\;\; 2)$ and $\lambda_{i,3}$ is
integral for all $i$. The identities for $\tilde s_{j,k}$ give
$$\tilde S=\left(\begin{matrix} 1 & d_1 & d_2 & d_3\cr d_1 & \e_1\e_2 d_2 &
\e_1 & \e_2 d_3\cr d_2 & \e_1 & \e_2 d_1 & \e_1 \e_2 d_3\cr d_3 &
\e_2 d_3 & \e_1 \e_2 d_3 & \tilde s_{3,3}\end{matrix}\right).$$
Since $\sigma(d_3)=\lambda_{3,1}=\e_2 d_3/d_1$ and
$\sigma^2(d_3)=\e_1 \e_2 d_3/d_2$, we must have $\sigma(d_3)\ne
d_3$. (Otherwise $\e_1=\e_2=d_1=d_2=1$ which fails.) Thus $d_3$ is
a root of an irreducible cubic $g(x)=x^3-c_1x^2+c_2 x-c_3$ and
$\e_2 d_1$ and $\e_1\e_2 d_2$ are ratios of roots of $g$. If $g$
had Galois group $S_3$, then the ratios of the roots of $g$ would
be roots of an irreducible sextic. Thus $g$ has Galois group
$\Z/3\Z$ and $G=\{1,\sigma,\sigma^2\}$. Note that
$$c_1=d_3+\lambda_{3,1}+\lambda_{3,2}=
\frac{d_3}{d_1d_2} (d_1d_2+\e_2 d_2+\e_1\e_2 d_1),$$
$$c_2=d_3\lambda_{3,1}+\lambda_{3,1}\lambda_{3,2}+\lambda_{3,2}d_3=
\frac{d_3^2}{d_1d_2} (\e_2 d_2+\e_1+\e_1\e_2 d_1),\;\; {\rm and}$$
$$c_3=d_3 \lambda_{3,1} \lambda_{3,2}= \e_1 \frac{d_3^3}{d_1d_2}.$$
Orthogonality of the columns of $\tilde S$ gives
$$\frac{c_1}{c_3}=\frac{\e_2d_1+\e_1\e_2d_2+\e_1 d_1d_2}{d_3^2} =-1,\;\;
{\rm and}$$
$$\frac{c_2}{c_3}=\frac{1+\e_2 d_1+\e_1 \e_2 d_2}{d_3} =
-\frac{\tilde s_{3,3}}{d_3}=-\lambda_{3,3}\in \Z.$$ Thus
$g(x)=x^3-c x^2 + n c x + c$ for integer $n, c$. Since $g$ has
Galois group $\Z/3\Z$,
$$\delta^2 = \frac{1}{c^2} {\rm discr}(g) = (n^2+4)c^2 - 2n(2n^2+9)c-27$$
must be a square. The resulting $n_{i,j,k}$ are
\begin{equation*}\begin{split}
n_{1,1,1} & = \frac{\e_2}{2}(\delta-nc-1)-\frac{\e_2\delta}{n^2+3},\\
n_{1,1,2} & = \frac{\e_1\e_2}{2(n^2+3)}(-\delta+nc-2n^2+3),\\
n_{1,1,3} & = \frac{1}{2(n^2+3)}(-n\delta+(n^2+2)c+3n),\\
n_{1,2,2} & = \frac{\e_2}{2(n^2+3)}(\delta+nc-2n^2+3),\\
n_{1,2,3} & = \frac{\e_1}{n^2+3}(3n-c),\\
n_{1,3,3} & = \frac{\e_2}{2(n^2+3)}(\delta-nc+2n^2-3),\\
n_{2,2,2} & = -\frac{\e_1\e_2}{2}(\delta+nc+1)+\frac{\e_1\e_2\delta}{n^2+3},\\
n_{2,2,3} & = \frac{1}{2(n^2+3)}(n\delta+(n^2+2)c+3n),\\
n_{2,3,3} & = \frac{\e_1\e_2}{2(n^2+3)}(-\delta-nc+2n^2-3),\\
n_{3,3,3} & = \frac{c+n^3}{n^2+3}.
\end{split}\end{equation*}
Integrality of $n_{1,2,3}$ requires $c\equiv 3n$ (mod $n^2+3$). If
we write $c=3n+a(n^2+3)$ for integer $a$, then we compute
$\delta^2=(n^2+3)^2(a^2(n^2+4)+2an-3)$. Hence
$\delta=(n^2+3)\beta$ where $\beta$ is integral and
$\beta^2=a^2(n^2+4)+2an-3$. Note that in particular this forces
$a\ne 0$. Rewriting it as $\beta^2=(an+1)^2+4a^2-4$, we see that
$\beta\equiv an+1$ (mod 2). Thus we compute
\begin{equation*}\begin{split}
n_{1,1,1} & = \frac{\e_2}{2}((n^2+1)\beta-3n^2-an(n^2+3)-1),\\
n_{1,1,2} & = \frac{\e_1\e_2}{2}(-\beta+an+1),\\
n_{1,1,3} & = \frac{1}{2}(-n\beta+3n+a(n^2+2)),\\
n_{1,2,2} & = \frac{\e_2}{2}(\beta+an+1),\\
n_{1,2,3} & = - \e_1 a,\\
n_{1,3,3} & = \frac{\e_2}{2}(\beta-an-1),\\
n_{2,2,2} & = -\frac{\e_1\e_2}{2}((n^2+1)\beta+3n^2+an(n^2+3)+1),\\
n_{2,2,3} & = \frac{1}{2}(n\beta+3n+a(n^2+2)),\\
n_{2,3,3} & = \frac{\e_1\e_2}{2}(-\beta-an-1),\\
n_{3,3,3} & = a+n,
\end{split}\end{equation*}
and these are all integral. Nonnegativity of these entries gives
further restrictions on the parameters. Looking at
$n_{1,2,2}+n_{1,3,3}=\e_2\beta$ we see that $\e_2$ is the sign of
$\beta$ (or $\beta=0$, but this gives $a=\pm 1$, $n=-a$, $c=a$ and
$g(x)=(x-a)(x^2-1)$ which is reducible). Looking at
$n_{1,1,2}+n_{2,3,3}=-\e_1\e_2\beta$, we see that $\e_1=-1$.
Looking at $n_{1,2,3}$ we see that $a>0$.
Nonnegativity provides additional constraints on the parameters,
but instead we look at the twist inequalities.

We saw above that $a>0$ hence $c=3n+a(n^2+3)\ge n^2+3n+3>0$. Thus
two of the roots $d_3$, $\lambda_{3,1}$ and $\lambda_{3,2}$ of $g$
must be positive. By symmetry, we may assume
$d_3>\lambda_{3,1}>0>\lambda_{3,2}$. Then $\e_2=1$ and we have
$$\tilde S=\left(\begin{matrix} 1 & d_3/\lambda_{3,1} & -d_3/\lambda_{3,2}
& d_3\cr d_3/\lambda_{3,1} & d_3/\lambda_{3,2} & -1 & d_3\cr
-d_3/\lambda_{3,2} & -1 & d_3/\lambda_{3,1} & - d_3\cr d_3 & d_3 &
- d_3 & nd_3\end{matrix}\right).$$ Let $M=\max(1/\lambda_{3,1},
1/\vert\lambda_{3,2}\vert)$ so that $M d_3=\max(d_1,d_2)$. Since
$D^2=(n^2+3)d_3^2$ and
$$\frac{1}{d_3^2}+\frac{1}{\lambda_{3,1}^2}+\frac{1}{\lambda_{3,2}^2}=
n^2+2,$$ the diagonal twist inequality coming from the $(0,0)$
entry gives
$$2 M^2\le n^2+3+\frac{\sqrt {n^2+3}}{d_3}.$$
This inequality allows only finitely many choices of the
parameters.

If $n>0$, then $g(-1/\phi_n)=-\phi_n^{-3}<0$ and therefore
$\lambda_{3,2}>-1/\phi_n$. Thus $M> \phi_n$ and we get
$$2\phi_n^2 < n^2+3+\frac{\sqrt {n^2+3}}{d_3}.$$
Further since $d_3\lambda_{3,1}=Mc$ and $d_3>\lambda_{3,1}$, we
have
$$d_3>{\sqrt {\phi_n c}}\ge {\sqrt {\phi_n (n^2+3n+3)}}.$$
For $n\ge 2$, we get a contradiction by noting that $\phi_n>n$,
hence these equations force
$$2n^2 < n^2 + 3 + \frac{1}{\sqrt n},$$
a contradiction. For $n=1$, plugging in gives a contradiction.

If $n=0$, then $\beta^2=4a^2-3$, hence $a=1$ and
$g(x)=x^3-3x^2+3$. Since $\tilde s_{3,3}=0$, the $(3,3)$ entry of
the twist equation gives $\theta_3^2 d_3^2
(1+\theta_1+\theta_2)=0$, hence $\theta_1=\bar \theta_2=e^{\pm
2\pi i/3}$. The $(0,3)$ and $(1,3)$ entries give
$$\theta_3(1+d_1 \theta_1 - d_2\theta_2)=D_+=
\theta_1\theta_3 (d_1-d_2\theta_1+\theta_2).$$ This case is
realized by $(A_1,7)_{\frac{1}{2}}$.

If $n<0$, then $1/M=\lambda_{3,1}$. One easily checks that $M$ and
$d_3$ are increasing functions of $c$, therefore it suffices to
check that the inequality $2M^2\le n^2+3 + (n^2+3)^{1/2} d_3^{-1}$
fails for $a=1$ and hence $c=n^2+3n+3$. The inequality fails for
$n\le -2$. (To see this simply compute both sides for $n=-2$. For
$n\le -3$, note that $g(-1/n)=-((n+1)/n)^3<0$ and
$g((n+1)^2)=-(n+1)^4+n+2<0$. Therefore $M< -n$ and $d_3>(n+1)^2$.
Hence we have $2M^2> 2n^2 > n^2+4$ and $(n^2+3)^{1/2}/d_3<1$, but
these combine to contradict the inequality.) For $n=-1$, $a=1$,
the inequality holds, but the resulting polynomial
$g(x)=x^3-x^2-x+1$ is reducible. Moving up to the next case
$n=-1$, $a=3$ the inequality fails. Thus there are no solutions in
this case.

\vspace{.5cm}

With the cases above completed, we consider $G$ which is not
transitive and contains no 3-cycle. Up to symmetry, it follows
that $G$ must be a subgroup of $\Z/2\Z \times \Z/2\Z =\langle
(0\;\;1), (2\;\; 3)\rangle$.

\vspace{.5cm}

\noindent {\bf Case 4.} $G$ contains the transposition
$\sigma=(2\;\; 3)$.

\vspace{.5cm}

In this case the parity condition gives $\e_1\e_2\e_3=-1$. Three
instances of the usual identity give $d_2=\tilde s_{0,2}=\e_2
\tilde s_{0,3}=\e_2 d_3$, $d_3=\tilde s_{0,3}=\e_3 \tilde
s_{0,2}=\e_3 d_2$, and $d_1=\tilde s_{0,1} = \e_1 \tilde
s_{0,1}=\e_1 d_1$. Since the $d_i$ are positive we conclude
$\e_1=\e_2=\e_3=1$, a contradiction. Thus we are left with only
three possibilities. Either $G=\Z/2\Z=\langle (0\;\; 1)(2\;\;
3)\rangle$, $G=\Z/2\Z=\langle (0\;\; 1)\rangle$, or $G$ is
trivial.

\vspace{.5cm}

\noindent {\bf Case 5.} $G$ contains $\sigma=(0\;\; 1)(2\;\; 3)$.

\vspace{.5cm}

Using the identities $\tilde s_{j,k}=\e_{\sigma(j)}\e_k \tilde
s_{\sigma(j),\sigma^{-1}(k)}$ and $\e_1\e_2\e_3=1$ gives
$$\tilde S=\left(\begin{matrix} 1 & d_1 & d_2 & d_3\cr d_1 & \e_1 & \e_2
d_3 & \e_1\e_2 d_2\cr d_2 & \e_2 d_3 & \tilde s_{2,2} & \tilde
s_{2,3}\cr d_3 &  \e_1\e_2 d_2 & \tilde s_{2,3} & \e_1 \tilde
s_{2,2}\end{matrix} \right).$$ Suppose first that $\e_1=1$. Then,
orthogonality of the first two columns of $\tilde S$ forces
$\e_2=-1$ and $d_1=d_2 d_3$. Orthogonality of the last column with
the first three gives the equations
$$\tilde s_{2,2} d_2 + \tilde s_{2,3} d_3 = d_2 (d_3^2-1),\;
\tilde s_{2,2} d_3 + \tilde s_{2,3} d_2 = d_3 (d_2^2-1),\; {\rm
and}\; \tilde s_{2,2} \tilde s_{2,3}=-d_2d_3.$$ Looking at the
cases $d_2=d_3$ and $d_2\ne d_3$ separately, these solve to give
$$\tilde s_{2,2} = -1\;\; {\rm and }\;\; \tilde s_{2,3} = d_1=d_2 d_3.$$
Hence
$$\tilde S=\left(\begin{matrix} 1 & d_2 d_3 & d_2 & d_3\cr d_2 d_3 & 1 &
-d_3 & -d_2\cr d_2 & -d_3 & -1 & d_2 d_3\cr d_3 & -d_2 & d_2 d_3 &
-1\end{matrix}\right).$$ This is exactly the $\tilde S$ matrix of
Case 2 above. Exactly as in that case, we get $d_2=\phi_m$,
$d_3=\phi_n$, and $d_1=d_2d_3$. The $N_i$ matrices and the twist
inequalities are the same, hence we conclude $(m,n)=(0,1)$,
$(0,2)$, or $(1,1)$. (Here we exclude $m=n=0$ since it gives $G$
trivial.) In the first case, $(m,n)=(0,1)$, the possible twist
matrices are given by $\theta_3=e^{\pm 4\pi i/5}$, $\theta_2=\pm
i$, and $\theta_1=\theta_2\theta_3$. In the second case,
$(m,n)=(0,2)$, no twist matrix exists. (To see this, note that the
$(1,1)$ and $(3,3)$ entries in the twist equation give
\begin{equation*}\begin{split}
\theta_1^2(\phi_2^2+\theta_1 + \theta_2 \phi_2^2 +\theta_3 ) & = D_+\\
-\theta_3^2(\phi_2^2+\theta_1 + \theta_2 \phi_2^2 +\theta_3 ) & =
D_+.
\end{split}\end{equation*}
Thus $\theta_1=\pm i \theta_3$. The $(0,1)$ and $(0,3)$ entries
give
\begin{equation*}\begin{split}
\theta_1(1+\theta_1 - \theta_2 -\theta_3 ) & = D_+\\
\theta_3(1-\theta_1 + \theta_2 -\theta_3 ) & = D_+.
\end{split}\end{equation*}
Subtracting these and using the result above gives $\theta_2=\pm
i$. Plugging these equations into $1+\theta_1\phi_2^2 +\theta_2 +
\theta_3\phi_2^2=D_+$, gives $D_+=(1\pm i)(1+\theta_3\phi_2^2)$.
Equating squared norms gives $\theta_3+\bar \theta_3 =
1-\phi_2^2$. However $1-\phi_2^2=-2-2{\sqrt 2}<-2$, so this is
impossible.) In the third case, $(m,n)=(1,1)$, the possible twist
matrices are given by $\theta_2=e^{\pm 4\pi i/5}$,
$\theta_3=e^{\pm 4\pi i/5}$, and $\theta_1=\theta_2\theta_3$.

Next consider the case $\e_1=-1$ so
$$\tilde S=\left(\begin{matrix} 1 & d_1 & d_2 & d_3\cr d_1 & -1 & \e_2
d_3 & -\e_2 d_2\cr d_2 & \e_2 d_3 & \tilde s_{2,2} & \tilde
s_{2,3}\cr d_3 & -\e_2 d_2 & \tilde s_{2,3} & -\tilde
s_{2,2}\end{matrix} \right).$$ By symmetry under interchanging
$N_2$ and $N_3$, we may assume $\e_2=+1$. Since $\sigma$ is the
only nontrivial element of the Galois group, we conclude that
$$d_1-\frac{1}{d_1}=n,\;\; d_2+\frac{d_3}{d_1}=r,\;\; {\rm and}\;\;
d_3-\frac{d_2}{d_1}=s$$ are integers. Hence $d_1=\phi_n$. If
$n<0$, then $d_1<1$ and since it is the largest root $d_3<d_2$.
But then it follows that ${\rm
trace}(N_1)=d_1-1/d_1+d_3/d_2-d_2/d_3<0$, an impossibility.
Further if $n=0$, then the same argument shows $d_1=1$ and
$d_3=d_2$. However, looking at the eigenvalues of $N_2$ shows
$\sigma(d_2)=d_3/d_1>0$ and looking at $N_3$ shows
$\sigma(d_3)=-d_2/d_1<0$. Thus $n\ge 1$ and $d_1$ is irrational.
Since $G=\Z/2\Z$, it follows that $K={\mathbb Q}[d_1]$. Hence we
can write $d_2=ad_1+b$ and $d_3=\tilde a d_1+\tilde b$ for
rational $a,\tilde a,b,\tilde b$. Hence
$$r=d_2+\frac{d_3}{d_1}=(a+\tilde b) d_1 + b + \tilde a - n \tilde b,\;\;
{\rm and}\;\; s=d_3-\frac{d_2}{d_1}=(\tilde a-b) d_1 +\tilde b - a + n
b.$$ Hence $\tilde b=-a$, $\tilde a=b$, $r=na+2b$ and $s=nb-2a$.
Note that in complex terms this gives $d_2+id_3 = (a+ib)(d_1-i)$
and $r+is=(a+ib)(n-2i)$. In particular
$D^2=1+d_1^2+d_2^2+d_3^2=(1+d_1^2)(n^2+r^2+s^2+4)/(n^2+4)$.

Since the columns of $\tilde S$ are of equal length $\tilde
s_{2,2}^2+\tilde s_{2,3}^2=1+d_1^2$. Since $\tilde s_{2,2}/d_2$ is
an eigenvalue of $N_2$, $(\tilde s_{2,2}/d_2) \sigma(\tilde
s_{2,2}/d_2) =\tilde s_{2,2}\tilde s_{2,3}/(d_2d_3)$ is an
integer. Further $\tilde s_{2,2}\ne 0$, since $\tilde s_{2,2}=0$
would force $\sigma(\tilde s_{2,2})=0$ and hence $\tilde
s_{2,3}=0$. Thus
$$\frac{1+d_1^2}{2d_2d_3}\ge \frac{\vert \tilde s_{2,2}\tilde s_{2,3}\vert}{d_2d_3}\ge 1.$$
The twist inequality coming from the $(0,1)$ entry of $\tilde S$
gives
\begin{equation*}\begin{split}
(1+d_1^2)^{1/2}(1+a^2+b^2)^{1/2}=D & \le 2 + 2\frac{d_2d_3}{d_1}\\
& \le 2 + \frac{2d_2d_3}{1+d_1^2}\cdot \frac{1+d_1^2}{d_1}\\
& \le 2+\frac{1+d_1^2}{d_1} = \frac{(1+d_1)^2}{d_1}.
\end{split}\end{equation*}
Rewriting this gives using
$$r^2+s^2 \le (n^2+4) \frac{4d_1^3+5d_1^2+4d_1+1}{d_1^2(1+d_1^2)}.$$
The twist inequality coming from the $(0,0)$ entry of $\tilde S$
is $2 d_1^2\le D^2 + D$, hence
$$2 d_1^2 \le \frac{(1+d_1)^4}{d_1^2}+\frac{(1+d_1)^2}{d_1},\;\;{\rm or}$$
$$d_1^4\le 5 d_1^3 + 8 d_1^2 + 5 d_1 + 1.$$
It follows that $d_1<7$, hence $1\le n\le 6$. Together with the
bound on $r^2+s^2$ above, this leaves only finitely many
possibilities (110 of them, insisting that $d_2$ and $d_3$ be
positive, but finite).

Since $G=\Z/2\Z$, the quantities
$$\frac{d_2d_3}{d_1}=\frac{r^2+nrs-s^2}{n^2+4},$$
$$\frac{d_3}{d_2}-\frac{d_2}{d_3}=n+\frac{(n^2+4)rs}{r^2+nrs-s^2},$$
$$\frac{\tilde s_{2,2}}{d_2}+\frac{\tilde s_{2,3}}{d_3}=
\frac{(n^2-4)r^3-12nr^2s-3(n^2-4)rs^2+4ns^3}{(r^2+s^2)(r^2+nrs-s^2)},$$
$$\frac{\tilde s_{2,3}}{d_2}-\frac{\tilde s_{2,2}}{d_3}=
\frac{4nr^3+3(n^2-4)r^2s-12nrs^2-(n^2-4)s^3}{(r^2+s^2)(r^2+nrs-s^2)},$$
$$\frac{\tilde s_{2,2}\tilde s_{2,3}}{d_2d_3} =
\frac{(4-3n^2)(r^4-6r^2s^2+s^4)-2n(n^2-12)rs(r^2-s^2)}{(r^2+s^2)^2(r^2+nrs-s^2)}
$$
are all integers. Only 6 of the 110 examples pass these
integrality conditions. These are $(n,r,s)=(1,2,1)$, $(1,3,-1)$,
$(2,2,2)$, $(3,2,3)$, $(4,2,4)$, and $(4,3,1)$. The cases with
$n=s$ and $r=2$ can be ignored since they give $a=0$ and $b=1$,
hence $d_2=1$, $d_3=d_1$, and $\tilde s_{2,2}=-1$. Thus $\tilde
s_{3,3}=1$. Invoking the symmetry under interchanging $N_1$ and
$N_3$ puts us back in the case $\e_1=1$. These are just the
$(0,n)$ examples discussed above. The remaining two examples fail
to give integral $N_i$ matrices and also fail the twist
inequalities. Thus there are no new examples in this case.

\vspace{.5cm}

\noindent {\bf Case 6.} $G$ contains the transposition
$\sigma=(0\;\; 1)$.

\vspace{.5cm}

Since we can exclude the cases above, $\sigma$ must be the only
non-trivial element of $G$. Up to symmetry there are two cases.
The parity condition gives $\e_1\e_2\e_3=-1$. Since $d_2=\tilde
s_{0,2}=\e_1\e_2\tilde s_{1,2}$ and $\tilde s_{1,2}=\e_2 \tilde
s_{0,2}=\e_2 d_2$, we conclude $\e_1=1$ and $\e_3=-\e_2$. Then
$\tilde s_{2,3}=\e_2\e_3\tilde s_{3,2}$ and the fact that $\tilde
S$ is symmetric forces $\tilde s_{2,3}=\tilde s_{3,2}=0$. By
symmetry we may assume $\e_2=1$ and $\e_3=-1$. Thus we get
$$\tilde S=\left(\begin{matrix}1 & d_1 & d_2 & d_3\cr d_1 & 1 & d_2 &
-d_3\cr d_2 & d_2 & \tilde s_{2,2} & 0\cr d_3 & -d_3 & 0 & \tilde
s_{3,3}\end{matrix}\right).$$ Note that $\lambda_{2,2}=\tilde
s_{2,2}/d_2$ and $\lambda_{3,3}=\tilde s_{3,3}/d_3$ are integers.
Note that orthogonality of the third column of $\tilde S$ with the
first forces $\lambda_{2,2}<0$. Since $d_1\ge \sigma(d_1)=1/d_1$,
we conclude $d_1\ge 1$. Thus orthogonality of the first and fourth
columns of $\tilde S$ gives $\lambda_{3,3}\ge 0$. It is
straightforward, though somewhat tedious, to build the $N_i$
matrices in this case and worry about their integrality and
nonnegativity; however, there is an easier approach. Since $\tilde
s_{2,3}=\tilde s_{3,2}=0$, the twist equation for the $(2,3)$
entry becomes
$$\theta_2\theta_3(d_2 d_3 -\theta_1 d_2 d_3) = 0.$$
Thus we conclude $\theta_1=1$. Using this fact the $(2,2)$ entry
becomes
$$\theta_2^2(2 d_2 ^2 + \theta_2 \tilde s_{2,2}^2) = D_+ \tilde s_{2,2}.$$
Since $\tilde s_{2,2}/d_2=\lambda_{2,2}<0$ is integral and $\vert
D_+\vert^2=D^2 = (\lambda_{2,2}^2+2)d_2^2$, equating the squared
norms of the sides of this equation gives
$$\theta_2+\bar \theta_2 = 1-\frac{2}{\lambda_{2,2}^2}.$$
The left hand side is an algebraic integer, hence we conclude
$\lambda_{2,2}=-1$ and $\theta_2=\exp(\pm 2\pi i/3)$. Similarly,
the $(3,3)$ entry gives $\lambda_{3,3}=1$ and $\theta_3=\exp(\pm
2\pi i/3)$. Equating the squared lengths of the last two columns
of $\tilde S$ now gives $d_2^2=d_3^2$, hence $d_2=d_3$. This is a
contradiction, since $\sigma(d_2)=\lambda_{2,1}=d_2/d_1$ but
$\sigma(d_3)=\lambda_{3,1}=-d_3/d_1$.

\vspace{.5cm}

\noindent {\bf Case 7.} $G$ is trivial.

\vspace{.5cm}

This case is also contained in \cite{CZ}.

In this case all the $d_i$ and $\lambda_{i,j}$ are integral. By
symmetry, we may assume $1\le d_1\le d_2\le d_3$. Since every
column of $\tilde S$ must have squared length $D^2$, we see that
$D^2$ must be a multiple of $d_i^2$ for all $i$. If $d_1=d_2=d_3$,
then $d_1^2$ must divide $D^2=3d_1^2+1$. Hence $d_1=d_2=d_3=1$. Up
to symmetry orthogonality of the columns of $\tilde S$ forces
$$\tilde S=\left(\begin{matrix} 1 & 1 & 1 & 1\cr 1 & -1 & -1 & 1\cr 1 & -1
& 1 & -1\cr 1 & 1 & -1 & -1\end{matrix}\right),$$ or
$$\tS=\left(\begin{matrix} 1 & 1 & 1 & 1\cr 1 & 1 & -1 & -1\cr 1 & -1
& 1 & -1\cr 1 & -1 & -1 & 1\end{matrix}\right).$$ For the first
$\tS$ matrix, this gives integral $N_i$ matrices
$$N_1=\left(\begin{matrix} 0 & 1 & 0 & 0\cr 1 & 0 & 0 & 0\cr 0 & 0
& 0 & 1\cr 0 & 0 & 1 & 0\end{matrix}\right),\;\;
N_2=\left(\begin{matrix} 0 & 0 & 1 & 0\cr 0 & 0 & 0 & 1\cr 1 & 0 &
0 & 0\cr 0 & 1 & 0 & 0\end{matrix}\right),\;\; {\rm and}\;\;
N_3=\left(\begin{matrix} 0 & 0 & 0 & 1\cr 0 & 0 & 1 & 0\cr 0 & 1 &
0 & 0\cr 1 & 0 & 0 & 0\end{matrix}\right).$$ The possibilities for
the corresponding twist matrix are
$$T=\left(\begin{matrix} 1 & 0 & 0 & 0\cr 0 & i & 0 & 0\cr 0 & 0
& 1 & 0\cr 0 & 0 & 0 & -i\end{matrix}\right),\;\;
\left(\begin{matrix} 1 & 0 & 0 & 0\cr 0 & i & 0 & 0\cr 0 & 0 & -1
& 0\cr 0 & 0 & 0 & i\end{matrix}\right),$$ or their complex
conjugates. (Note that this is the case $m=n=0$ of the form found
in Cases 2 and 5.) For the second $\tS$ matrix, the compatible $T$
matrices are listed in Table \ref{rank4ms}.

If the $d_i$ are not all equal, then $d_1<d_3$, hence
$1+d_1^2<d_3^2$ and $d_3^2<D^2<3 d_3^2$. Thus we must have $D^2=2
d_3^2$, i.e., two of the $\tilde s_{i,3}$ are zero and the other
two are $\pm d_3$. Of course $\tilde s_{0,3}=d_3$. Suppose $\tilde
s_{j,3}=\pm d_3$ is the other nonzero entry in the last column.
Orthogonality of the first and last columns gives $d_3(1\pm
d_j)=0$, hence the lower sign is correct and $d_j=1$. Thus using
symmetry we may assume $j=1$. Orthogonality of the remaining
columns with the last column gives $\tilde s_{1,1}=d_1=1$ and
$\tilde s_{2,1}=d_2$. Orthogonality of the third column and the
first two gives $\tilde s_{2,2}=-2$ and
$$\tilde S=\left(\begin{matrix} 1 & 1 & d_2 & d_3\cr 1 & 1 & d_2 & -d_3\cr
d_2 & d_2 & -2 & 0\cr d_3 & -d_3 & 0 & 0\end{matrix}\right).$$
However equality of the squared lengths of the last two columns
now gives $D^2=2d_3^2=2d_2^2+4$ or $d_3^2=d_2^2+2$, an
impossibility.

\end{proof}

\vspace{.5cm}

\section{Realization of fusion rules and Classification of MTCs}\label{realization}

In this section, for each modular fusion rule $(\mN;\tS)$ in
Theorems \ref{Srank2}, \ref{Srank3}, \ref{Srank4}, we will first
determine all modular symbols with this fusion rule which also
satisfy (1)(2) of Proposition \ref{modularprop}; then classify all
MTCs realizing each such modular symbol.

For each modular fusion rule $(\mN;\tS)$, there are two choices of
compatible $S$ matrices: $S=\frac{1}{D} \tS$ or $-\frac{1}{D}\tS$.
When the two modular symbols are realized by (2+1)-TQFTs,
respectively, one TQFT is obtained from the other by tensoring the
trivial theory with $S=(-1)$.  The quantum invariant of the
3-sphere will be $\frac{1}{D}$ or $-\frac{1}{D}$, respectively.
Also the topological central charge $c$ of the two theories will
differ by $4$.  Another symmetry for modular symbols is complex
conjugation: to change $(\mN;S,T)$ to
$(\mN;S^{\dagger},T^{\dagger})$.  Complex conjugation of a modular
symbol gives rise to a different modular symbol if one of the
$S,T$ is not a real matrix.

Given an $\tS$ matrix, we can obtain all fusion matrices by using
the Verlinde formulas.  Instead of listing fusion matrices, we
will present them as fusion rules.  In the following, we will not
list trivial fusion rules such as $1\otimes x=x$ and those that
can be obtained from obvious identities such as $x\otimes
y=y\otimes x$. We will also write $x\otimes y$ as $xy$ sometimes.
Then we use relations (3)(i)-(iv) of Definition \ref{fusionrule}
together with (1) (2) of Proposition \ref{modularprop} to
determine the possible $T$-matrices.  As Example \ref{isingex}
illustrates, Proposition \ref{modularprop} is necessary to get
finitely many solutions in some cases.  We find that there are
finitely many modular symbols $(\mN;S,T)$ of rank$\leq4$
satisfying Proposition \ref{modularprop}. Modulo the symmetry
$S\rightarrow -S$, these modular symbols are classified in Table
\ref{rank4ms}.  In the table, $\zeta_m=e^{\frac{2\pi i}{m}}$.  The labels will be $\{1,X \}$ for rank=$2$,
$\{1,X,Y \}$ for rank=$3$, and $\{1,X,Y,Z\}$ for rank=$4$.  They
will correspond to rows $1,2,3,4$ of the $\tS$ matrices.  The $\#$
is the number of modular symbols satisfying Proposition
\ref{modularprop} modulo the symmetry $S\rightarrow -S$.  The
column $P$ stands for primality, and the column $G$ is the Galois
group of the modular fusion rule.

\begin{table}\caption{Rank$\leq4$ unitary modular symbols}\label{rank4ms}
\begin{tabular}{|c|p{4cm}|p{5cm}|c|c|c|}
\hline
$\tS$ matrix & Fusion rules & $T$ matrix & $\#$s & P & G\\
\hline\hline
$\tS=(1)$ &  & $(1)$ & 1 & Yes & 1\\
\hline
 Thm 3.1(1) & $X^2=1$ & $\Diag(1,\pm i)$  & 2 & Yes &1\\
\hline Thm 3.1(2) & $X^2=1+X$ & $\Diag(1,(\zeta_5)^{\pm 2})$ & 2  & Yes& $\Z_2$\\
\hline Thm 3.2(1) &\raggedright $X^2=X^*$, $XX^*=1$, $(X^*)^2=X$ & \raisebox{-1ex}[0pt]{$\Diag(1,\zeta_3^{\pm 1},\zeta_3^{\pm 1})$} & 2 & Yes& $\Z_2$\\
\hline Thm 3.2(2) &\raggedright $X^2=1+Y$, $XY=X$, $Y^2=1$ & \raisebox{-1ex}[0pt]{$\Diag(1,(\zeta_{16})^{2k+1},-1)$} & 8 & Yes& $\Z_2$\\
\hline Thm 3.2(3) &\raggedright $X^2=1+X+Y$, $XY=X+Y$, $Y^2=1+X$ & \raisebox{-1ex}[0pt]{$\Diag(1,(\zeta_7)^{\pm 5},(\zeta_7)^{\pm 1})$} & 2 & Yes& $\Z_3$\\
\hline Thm 4.1(1) &\raggedright $X^2=Y=(X^*)^2$, $XX^*=1=Y^2$, $XY=X^*$, $X^*Y=X$   &\raggedright \raisebox{-1ex}[0pt]{$\Diag(1,-1,(\zeta_8)^{\pm m},(\zeta_8)^{\pm m}),$} \\ \raisebox{-1ex}[0pt]{$m$=1,3}& 4 &Yes & $\Z_2$\\
\hline Thm 4.1(2) &\raggedright $X^2=1$, $XY=Z$, $XZ=Y$, $Y^2=1$,
$YZ=X$, $Z^2=1$ &
\raggedright \raisebox{-1ex}[0pt]{$\Diag(1,-1,\e_1,\e_1)$,} \\ \raisebox{-1ex}[0pt]{$\e_1^2=1$}  & 2 & Yes& 1\\
\hline Thm 4.1(3) &\raggedright $X^2=1$, $XY=Z$, $XZ=Y$, $Y^2=1$, $YZ=X$, $Z^2=1$ &  \raggedright
\raisebox{-1ex}[0pt]{$\Diag(1,\theta_1,\theta_2,\theta_1\theta_2)$,}\\ \raisebox{-1ex}[0pt]{$\theta_i^2=-1$} &3 & No& 1\\
\hline Thm 4.1(4) &\raggedright $X^2=1+X$, $XY=Z$, $XZ=Y+Z$, $Y^2=1$, $YZ=X$, $Z^2=1+X$ & \raggedright
\raisebox{-2ex}[0pt]{$\Diag(1,\theta_1,\theta_2,\theta_1\theta_2)$,}\\ \raisebox{-2ex}[0pt]{$\theta_1=(\zeta_5)^{\pm 2}$, $\theta_2=\pm i$} & 4 & No& $\Z_2$\\
\hline Thm 4.1(5) &\raggedright $X^2=1+X$, $XY=Z$, $XZ=Y+Z$, $Y^2=1+Y$, $YZ=X+Z$, $Z^2=1+X+Y+Z$ & \raggedright \raisebox{-3ex}[0pt]{$\Diag(1,\theta_1,\theta_2,\theta_1\theta_2)$,} \raisebox{-3ex}[0pt]{$\theta_1=(\zeta_5)^{\pm 2}$, $\theta_2=(\zeta_5)^{\pm 2}$} & 3 & No&$\Z_2$\\
\hline Thm 4.1(6) &\raggedright $X^2=1+X+Y$, $XY=X+Y+Z$, $XZ=Y+Z$,
$Y^2=1+X+Y+Z$, $YZ=X+Y$, $Z^2=1+X$ &\raggedright \raisebox{-4.5ex}[0pt]{$\Diag(1,\zeta_9^{\pm 2}, (\zeta_9)^{\pm 6},(\zeta_9)^{\mp 6})$} & 2 & Yes& $\Z_3$\\
\hline
\end{tabular}
\end{table}

With modular symbols determined, we turn to realizing each of them
with MTCs.   First let us consider the $\tS$-matrices
corresponding to Theorem \ref{Srank4}(3)-(5).  In each of these
cases there is a rank$=$2 tensor subcategory corresponding to the
objects labelling columns $1$ and $3$ of the $\tS$ matrix. Further
inspection shows that the submatrix of $\tS$ corresponding to rows
and columns $1$ and $3$ is invertible.
It is obvious that the tensor subcategory generated by the trivial
object and the object labelling column $3$ is a modular
subcategory.  For the $\tS$-matrices of Theorem
\ref{Srank4}(3),(4) these rank$=$2 modular subcategories are
equivalent to the UMTCs corresponding to
Theorem \ref{Srank2}(1), while the modular subcategory
corresponding to the $\tS$-matrix of Theorem \ref{Srank4}(5) is
equivalent to (one of) the UMTCs coming from
Theorem \ref{Srank2}(2).  By \cite{MugerLMS}[Theorem 4.2] this
implies that the MTCs corresponding to these $\tS$ matrices are
direct products of rank$=$2 MTCs.  For this reason we will not
write down realizations or complete data for these MTCs as they
can be deduced from their product structure.

MTCs realizing the remaining $8$ nontrivial modular symbols are
\emph{prime}, i.e. they do not have non-trivial modular
subcategories.  To complete the classification, we need to solve
the pentagon and hexagon equations for all $8$ modular symbols.
The solutions of the pentagon equations are organized into the
$F$-matrices whose entries are called $6j$ symbols.  The solutions
of hexagons are given by the braiding eigenvalues.

\subsection{F-matrices}\label{fmatrix}

Given an MTC $\mC$.  A $4$-punctured sphere $S^2_{a,b,c,d}$, where
the 4 punctures are labelled by $a,b,c,d$, can be divided into two
pairs of pants(=3-punctured spheres) in two different ways.  In
the following figure, the 4-punctured sphere is the boundary of a
thickened neighborhood of the graph in either side, and the two
graphs encode the two different pants-decompositions of the
4-punctured sphere.  The F-move is just the change of the two
pants-decompositions.

 When bases of all pair of pants spaces $\Hom(a\otimes b,c)$ are chosen, then the
two pants decompositions of $S^2_{a,b,c,d}$ determine bases of the
vector spaces $\Hom((a\otimes b)\otimes c,d)$, and $\Hom(a\otimes
(b\otimes c),d)$, respectively. Therefore the $F$-move induces a
matrix $F^{a,b,c}_{d}: \Hom((a\otimes b)\otimes c,d)\rightarrow
\Hom(a\otimes (b\otimes c),d)$, which are called the F-matrices.
Consistency of the $F$ matrices are given by the pentagon
equations.

For each quadruple $(a,b,c,d)$, we have an $F$-matrix whose
entries are indexed by a pair of triples $((m,s,t),(n,u,v))$,
where $m,n$ are the labels for the internal edges, and $s,t,u,v$
are indices for a basis of the $\Hom(x\otimes y,z)$ spaces with
$\dim
>1$. For the MTCs in our paper, none of the $\Hom(a\otimes b,c)$
has $\dim>1$, so we will drop the $s,t,u,v$ from our notation.  If
one of the $a,b,c$ in $F^{a,b,c}_d$ is the trivial label, then we
may assume $F^{a,b,c}_d$ is the identity matrix.  But we cannot
always do so if $d$ is the trivial label.  In the following, any
unlisted $F$ matrix is the identity.

\[ \xy
(0,10)*{}="A"; (4,10)*{}="B";(8,10)*{}="C"; (12,-10)*{}="D";
(4,3.3)*{}="E";(8,-3.3)*{}="F"; "A";"D" **\dir{-}; "B";"E"
**\dir{-};"C";"F" **\dir{-}; (0,12)*{a};(4,12)*{b};
(8,12)*{c};(12,-12)*{d}; (16,0)*{}="S"; (20,0)*{}="S'"; "S";"S'"
**\dir2{-}; (18,4)*{F_d^{abc}}; (26,10)*{}="A'";
(30,10)*{}="B'";(34,10)*{}="C'"; (22,-10)*{}="D'";
(26,-3.3)*{}="E'";(30,3.3)*{}="F'"; "C'";"D'" **\dir{-}; "B'";"F'"
**\dir{-};"A'";"E'" **\dir{-}; (26,12)*{a};(30,12)*{b};
(34,12)*{c};(22,-12)*{d}; (4,-3)*{m}; (30,-3)*{n};
\endxy \]

\subsection{Braidings and twists}

The twist of the simple type $X_i$ will be denoted by $\theta_i$,
and it is defined by the following positive twist.

\[ \xy
(0,10)*{}="T"; (0,-10)*{}="B"; (0,5)*{}="T'"; (0,-5)*{}="B'";
"T";"T'" **\dir{-}; "B";"B'" **\dir{-}; (7,0)*{}="LB"; "B'";"LB"
**\crv{(1,4) & (7,4)}; \POS?(.25)*{\hole}="2z"; "LB"; "2z"
**\crv{(8,-6) & (2,-6)}; "2z"; "T'" **\crv{(0,-3)}; (12,0)*{}="S";
(14,0)*{}="S'"; "S";"S'" **\dir2{-}; (18,0)*{\theta_i};
(22,10)*{}="U"; (22,-10)*{}="V"; "U";"V" **\dir{-};
\endxy \]

The braiding eigenvalues are defined by the following diagram:

\[ \xy
(0,12)*{}="A"; (6,12)*{}="B"; (3,0)*{}="V"; (3,-10)*{}="C";
"V";"C" **\dir{-}; "V";"B" **\crv{(1,2) & (1,4)};
\POS?(.65)*{\hole}="2z"; "V";"2z" **\crv{(5,2) & (5,4)}; "2z";"A"
**\crv{(3,5) & (2,6)}; (0,14)*{a};(6,14)*{b};(3,-12)*{c};
(12,0)*{}="S"; (14,0)*{}="S'"; "S";"S'" **\dir2{-};
(20,0)*{R_c^{ab}}; (26,12)*{}="A'"; (32,12)*{}="B'";
(29,0)*{}="V'"; (29,-10)*{}="C'"; "V'";"C'" **\dir{-}; "V'";"A'"
**\dir{-}; "V'";"B'" **\dir{-};
(26,14)*{a};(32,14)*{b};(29,-12)*{c};
\endxy \]

The consistency equations of the braidings are given by two
independent families of hexagon equations.  If $c$ is the trivial
label, and label $a$ is self-dual, then
$R_{1}^{aa}=\nu_a\theta_a^{-1}$, where $\nu_a$ is the
Frobenius-Schur indicator of $a$.

\subsection{Explicit data}\label{explicitdata}

In this section, we give the explicit data for at least one
realization of each prime modular fusion rule.  Since each modular
symbol can have up to $4$ MTC realizations, we will present the
complete data for only one of them. We choose one with the
following properties:

(1) The $(0,0)$ entry of the $S$ matrix is $\frac{1}{D}$, where
$D$ is the total quantum order;

(2) In a category with a generating non-abelian simple object $X$,
we choose a theory with the positive exponent of the twist
$\theta_X$ being the smallest. This is inspired by anyon theory
that the simple object with the smallest exponent is the most
relevant in physical experiments.

If a modular symbol $(\mN;S,T)$ is realized by an MTC, then the
modular symbol $(\mN;-S,T)$ is also realized by an MTC.  The
modular symbol $(\mN;\tS^\dag,T^\dag)$ is realized by complex
conjugating all $F$ matrices and braidings of $(\mN;S,T)$. So in
the following each group of data will be for $4$ MTCs if any of
$S,T,F$ and braidings are not real; otherwise there will be two.
We choose the $F$ matrices to be unitary, and real if possible.

In anyon theory, labels will be called anyon types.  The smallest
positive exponent of a twist $\theta_i$ will be called the
topological spin of the anyon type $i$.  Topological spins are the
conformal dimensions modulo integers of the corresponding primary
field if the MTC has an corresponding RCFT.  The last line of the
data lists all quantum group realizations of the same theory. We
did not list the Frobenius-Schur indicators of anyons because they
can be calculated by the formula in Proposition \ref{modularprop}.
In the following data, only the semion $s$ and the $(A_1,2)$
non-abelian anyon $\sigma$ have Frobenius-Schur indicator=$-1$.

\subsubsection{Semion MTC}

We will use $s$ to denote the non-trivial label.

Anyon types: $\{1,s\}$

Fusion rules: $s^2=1$

Quantum dimensions: $\{1,1\}$

Twists: $\theta_1=1, \theta_{s}=i$

Total quantum order: $D=\sqrt{2}$

Topological central charge: $c=1$

Braidings: $R_1^{ss}=i$

S-matrix: $ S=\frac{1}{\sqrt{2}} \begin{pmatrix} 1 & 1 \\
1 & -1
\end{pmatrix},$

F-matrices:  $ F^{s,s,s}_s=(-1)$

Realizations: $(A_1,1)$, $(E_7,1)$.

\subsubsection{Fibonacci MTC}

We will use $\varphi$ to denote the golden ratio
$\varphi=\frac{1+\sqrt{5}}{2}$ and $\tau$ the non-trivial label.

Anyon types: $\{1,\tau\}$

Fusion rules: $\tau^2=1+\tau$

Quantum dimensions: $\{1,\varphi\}$

Twists: $\theta_1=1, \theta_{\tau}=e^{\frac{4\pi i}{5}}$

Total quantum order:
$D=2\cos(\frac{\pi}{10})=\frac{\sqrt{5}}{2\sin(\frac{\pi}{5})}$

Topological central charge: $c=\frac{14}{5}$

Braidings: $R_1^{\tau \tau}=e^{-\frac{4\pi i}{5}},
R_{\tau}^{\tau\tau}=e^{\frac{3\pi i}{5}}$

S-matrix: $ S=\frac{1}{\sqrt{2+\varphi}} \begin{pmatrix} 1 & \varphi \\
\varphi & -1
\end{pmatrix}$,

F-matrices:  $ F^{\tau,\tau,\tau}_{\tau}=\begin{pmatrix} \varphi^{-1} & {\varphi}^{-1/2} \\
{\varphi}^{-1/2} & -\varphi^{-1}
\end{pmatrix}$

Realizations: $(A_1,3)_{\frac{1}{2}}$, $(G_2,1)$, complex conjugate of $(F_4,1)$.

\subsubsection{$\Z_3$ MTC}  We will use $\omega$ for both a non-trivial label and the
root of unity $\omega=e^{2\pi i/3}$.  No confusions should arise.

Anyon types: $\{1,\omega,\omega^{*}\}$

Fusion rules:
$\omega^2=\omega^*,\omega\omega^*=1,(\omega^*)^2=\omega$

Quantum dimensions: $\{1,1,1\}$

Twists: $\theta_1=1,
\theta_{\omega}=\theta_{\omega^*}=e^{\frac{2\pi i}{3}}$

Total quantum order: $D=\sqrt{3}$

Topological central charge: $c=2$

Braidings:
$R^{\omega,\omega^*}_{1}=R^{\omega^*,\omega}_{1}=e^{-\frac{2\pi
i}{3}},
R^{\omega,\omega}_{\omega^*}=R^{\omega^*,\omega^*}_{\omega}=e^{-\frac{4\pi
i}{3}}$,

S-matrix: $S=\frac{1}{\sqrt{3}}\begin{pmatrix} 1 &1 &1\\ 1 &
\omega&\omega^2\\1&\omega^2&\omega\end{pmatrix}$,

F-matrices:  $ F^{a,b,c}_d=(1)$ for any $a,b,c,d$,

Realizations: $(A_2,1)$, $(E_6,1)$

\subsubsection{Ising MTC}

We will use $1,\sigma,\psi$ to denote the non-trivial labels.

Anyon types: $\{1,\sigma, \psi\}$

Fusion rules:
$\sigma^2=1+\psi,\sigma\psi=\psi\sigma=\sigma,\psi^2=1$

Quantum dimensions: $\{1,\sqrt{2}, 1\}$

Twists: $\theta_1=1, \theta_{\sigma}=e^{\frac{\pi i}{8}},\theta_{\psi}=-1$

Total quantum order: $D=2$

Topological central charge: $c=\frac{1}{2}$

Braidings: $R_1^{\sigma \sigma}=e^{-\frac{\pi i}{8}},
R_1^{\psi\psi}=-1,
R_{\sigma}^{\psi\sigma}=R_{\sigma}^{\sigma\psi}=-i,
R_{\psi}^{\sigma\sigma}=e^{\frac{3\pi i}{8}}$

S-matrix: $ S=\frac{1}{2}\begin{pmatrix} 1 &\sqrt{2} &1\\\sqrt{2}
& 0&-\sqrt{2}\\ 1&-\sqrt{2}&1\end{pmatrix}$,

F-matrices: $
F^{\sigma,\sigma,\sigma}_{\sigma}=\frac{1}{\sqrt{2}}\begin{pmatrix}
1 &1\\ 1&-1\end{pmatrix}$,$ F^{\psi,\sigma,\psi}_{\sigma}=(-1)$,$
F^{\sigma,\psi,\sigma}_{\psi}=(-1)$,

Realizations: complex conjugate of $(E_8,2)$.

\subsubsection{$(A_1,2)$ MTC}

We will use $1,\sigma,\psi$ to denote the non-trivial labels
again.

Anyon types: $\{1,\sigma, \psi\}$

Fusion rules:
$\sigma^2=1+\psi,\sigma\psi=\psi\sigma=\sigma,\psi^2=1$

Quantum dimensions: $\{1,\sqrt{2}, 1\}$

Twists: $\theta_1=1, \theta_{\sigma}=e^{\frac{3\pi
i}{8}},\theta_{\psi}=-1$

Total quantum order: $D=2$

Topological central charge: $c=\frac{3}{2}$

Braidings: $R_1^{\sigma \sigma}=-e^{-\frac{3\pi i}{8}},
R_1^{\psi\psi}=-1,
R_{\sigma}^{\psi\sigma}=R_{\sigma}^{\sigma\psi}=i,
R_{\psi}^{\sigma\sigma}=e^{\frac{\pi i}{8}}$

S-matrix: $ S=\frac{1}{2}\begin{pmatrix} 1 &\sqrt{2} &1\\\sqrt{2}
& 0&-\sqrt{2}\\ 1&-\sqrt{2}&1\end{pmatrix}$,

F-matrices: $
F^{\sigma,\sigma,\sigma}_{\sigma}=-\frac{1}{\sqrt{2}}\begin{pmatrix}
1 &1\\ 1&-1\end{pmatrix}$,$ F^{\psi,\sigma,\psi}_{\sigma}=(-1)$,$
F^{\sigma,\psi,\sigma}_{\psi}=(-1)$,

Realizations: $(A_1,2)$.

\subsubsection{$(A_1,5)_{\frac{1}{2}}$ MTC}

We will use $1, \alpha,\beta$ to denote the non-trivial labels.
Note that $1,\alpha,\beta$ are special labels for $1,Y,X$ in
Theorem 3.2(3) of Table \ref{rank4ms}.

Anyon types: $\{1,\alpha,\beta\}$

Fusion rules:
$\alpha^2=1+\beta,\alpha\beta=\alpha+\beta,\beta^2=1+\alpha+\beta$

Quantum dimensions: $\{1,d,d^2-1\}$, where
$d=2\cos(\frac{\pi}{7})$

Twists: $\theta_1=1, \theta_{\alpha}=e^{\frac{2\pi
i}{7}},\theta_{\beta}=e^{\frac{10\pi i}{7}}$

Total quantum order: $D=\frac{\sqrt{7}}{2\sin(\frac{\pi}{7})}$

Topological central charge: $c=\frac{48}{7}$

Braidings:  $R_1^{\alpha \alpha}=e^{-\frac{2\pi i}{7}},
R_1^{\beta\beta}=e^{-\frac{10\pi i}{7}},$

$R_{\alpha}^{\beta\beta}=e^{-\frac{2\pi i}{7}},
R_{\alpha}^{\alpha\beta}=R_{\alpha}^{\beta\alpha}=e^{\frac{9\pi
i}{7}},$

$R_{\beta}^{\beta\beta}=e^{-\frac{5\pi
i}{7}},R_{\beta}^{\alpha\beta}=R_{\beta}^{\beta\alpha}=e^{\frac{6\pi
i}{7}}, R_{\beta}^{\alpha\alpha}=e^{-\frac{4\pi i}{7}}$

S-matrix: $S=\frac{1}{D}\begin{pmatrix} 1 &d& d^2-1
\\ d & -d^2+1&1\\d^2-1&1&-d\end{pmatrix}$

F-matrices: see the end of this subsection.

Realizations: $(A_1,5)_{\frac{1}{2}}$

\subsubsection{$\Z_4$ MTC}

We will use $1,\e,\sigma,\sigma^*$ to denote the non-trivial labels.

Anyon types: $\{1,\e,\sigma,\sigma^*\}$

Fusion rule: $\e^2=\sigma\sigma^*=1, \sigma^2=(\sigma^*)^2=\e,\sigma\e=\sigma^*,\sigma^*\e=\sigma$

Quantum dimensions: $\{1,1,1,1\}$

Twists: $\theta_1=1,
\theta_{\e}=-1,\theta_{\sigma}=\theta_{\sigma^*}=e^{\frac{\pi
i}{4}}$

Total quantum order: $D=2$

Topological central charge: $c=1$

Braidings: $R^{\e,\e}_1=-1,
R^{\sigma,\sigma}_{\e}=R^{\sigma^*,\sigma*}_{\e}=e^{\frac{\pi
i}{4}}$,
$R^{\sigma,\sigma^*}_{1}=R^{\sigma^*,\sigma}_{1}=e^{-\frac{\pi
i}{4}}$,
$R^{\sigma,\e}_{\sigma^*}=R^{\e,\sigma}_{\sigma^*}=R^{\sigma^*,\e}_{\sigma}=R^{\e,\sigma^*}_{\sigma}=-i$

S-matrix:  $S=\frac{1}{2}\begin{pmatrix}
   1& 1& 1 & 1 \\
   1 & 1& -1 &-1\\
   1 & -1&-i &i\\
   1&-1&i& -i
 \end{pmatrix}$

F-matrices:
$F^{\sigma,\sigma,\sigma}_{\sigma^*}=F^{\sigma^*,\sigma^*,\sigma^*}_{\sigma}=F^{\e,\sigma,\e}_{\sigma}
=F^{\e,\sigma^*,\e}_{\sigma^*}=F^{\sigma,\e,\sigma^*}_{\e}=F^{\sigma^*,\e,\sigma}_{\e}=(-1)$

Realizations: $(A_3,1)$, $(D_9,1)$.

\subsubsection{Toric code MTC}

The fusion rules are the same as $\Z_2\times \Z_2$, but the theory
is not a direct product. We will use $1,e,m, \e$ to denote the
non-trivial labels.

Anyon types: $\{1,e,m,\e\}$

Fusion rules: $e^2=m^2=\e^2=1, em=\e,e\e=m,m\e=e$

Quantum dimensions: $\{1,1,1,1\}$

Twists: $\theta_1=\theta_e=\theta_m=1, \theta_{\e}=-1$

Total quantum order: $D=2$

Topological central charge: $c=0$

Braidings: $R^{\e,\e}_1=-1, R^{e,m}_{\e}=1,R^{m,e}_{\e}=-1$,
$R^{e,e}_{1}=R^{m,m}_{1}=1$,
$R^{\e,m}_{e}=1,R^{m,\e}_{e}=-1,R^{e,\e}_{m}=1, R^{\e,e}_{m}=-1$

S-matrix: $S=\frac{1}{2}\begin{pmatrix}
   1& 1& 1 &1 \\
   1 & 1& -1&-1  \\
   1 & -1 &1 &-1 \\
   1 &-1 & -1& 1
 \end{pmatrix}$

F-matrices: $F^{a,b,c}_d=(1)$ for all $a,b,c,d$.

Realizations: $(D_8,1)$,  $D(\Z_2)$---quantum double of $\Z_2$.

\subsubsection{$(D_4,1)$ MTC}

The fusion rules are the same as $\Z_2\times \Z_2$, but the theory
is not a direct product. We will use $1,e,m, \e$ to denote the
non-trivial labels again.

Anyon types: $\{1,e,m,\e\}$

Fusion rules: $e^2=m^2=\e^2=1, em=\e,e\e=m,m\e=e$

Quantum dimensions: $\{1,1,1,1\}$

Twists: $\theta_1=1, \theta_e=\theta_m=\theta_{\e}=-1$

Total quantum order: $D=2$

Topological central charge: $c=4$

Braidings: $R^{\e,\e}_1=-1, R^{e,m}_{\e}=-1,R^{m,e}_{\e}=1$,
$R^{e,e}_{1}=R^{m,m}_{1}=-1$,
$R^{\e,m}_{e}=1,R^{m,\e}_{e}=-1,R^{e,\e}_{m}=1, R^{\e,e}_{m}=-1$

S-matrix: $S=\frac{1}{2}\begin{pmatrix}
   1& 1& 1 &1 \\
   1 & 1& -1&-1  \\
   1 & -1 &1 &-1 \\
   1 &-1 & -1& 1
 \end{pmatrix}$

F-matrices: $F^{a,b,c}_d=(1)$ for all $a,b,c,d$.

Realizations: $(D_4,1)$.

\subsubsection{$(A_1,7)_{\frac{1}{2}}$ MTC}

We will use $1,\alpha,\omega,\rho$ to denote the non-trivial
labels.  Note that $1,\alpha,\omega,\rho$ are special labels for
$1,Z,Y,X$ in Theorem 4.1(6) of Table \ref{rank4ms}.

Anyon types: $\{1,\alpha, \omega,\rho\}$

Fusion rules:
$\alpha^2=1+\omega,\alpha\omega=\alpha+\rho,\alpha\rho=\omega+\rho,\omega^2=1+\omega+\rho,$

$\omega\rho=\alpha+\omega+\rho,\rho^2=1+\alpha+\omega+\rho$

Quantum dimensions: $\{1,d,d^2-1,d+1\}$, where
$d=2\cos(\frac{\pi}{9})$ and $d^3=3d+1$.

Twists: $\theta_1=1, \theta_{\alpha}=e^{\frac{2\pi i}{3}},
\theta_{\omega}=e^{\frac{4\pi i}{9}}, \theta_{\rho}=e^{\frac{4\pi
i}{3}}$

Total quantum order: $D=\frac{3}{2\sin(\frac{\pi}{9})}$

Topological central charge: $c=\frac{10}{3}$

Braidings: $R_1^{\alpha \alpha}=e^{-\frac{2\pi i}{3}},
R_1^{\omega\omega}=e^{-\frac{4\pi i}{9}},R_1^{\rho
\rho}=e^{-\frac{4\pi i}{3}}$

$R_{\alpha}^{\alpha\omega}=R_{\alpha}^{\omega\alpha}=e^{\frac{7\pi
i}{9}},
R_{\alpha}^{\omega\rho}=R_{\alpha}^{\rho\omega}=e^{\frac{4\pi
i}{9}},
R_{\alpha}^{\rho\rho}=-1$

$R_{\omega}^{\alpha\rho}=R_{\omega}^{\rho\alpha}=e^{\frac{2\pi
i}{9}},
R_{\omega}^{\omega\rho}=R_{\omega}^{\rho\omega}=e^{-\frac{2\pi
i}{3}}, R_{\omega}^{\alpha\alpha}=e^{\frac{5\pi i}{9}},
R_{\omega}^{\rho\rho}=e^{-\frac{\pi i}{9}},
R_{\omega}^{\omega\omega}=e^{\frac{7\pi i}{9}},$

$R_{\rho}^{\alpha\omega}=R_{\rho}^{\omega\alpha}=e^{-\frac{8\pi
i}{9}}, R_{\rho}^{\alpha\rho}=R_{\rho}^{\rho\alpha}=e^{-\frac{\pi
i}3}, R_{\rho}^{\rho\omega}=R_{\rho}^{\omega\rho}=e^{\frac{7\pi
i}{9}}, R_{\rho}^{\omega\omega}=e^{\frac{2\pi i}{9}},
R_{\rho}^{\rho\rho}=e^{-\frac{2\pi i}{3}}.$

S-matrix:  $S=\frac{1}{D}\begin{pmatrix}
   1& d& d^2-1 &d+1\\
   d & -d-1& d^2-1 & -1  \\
   d^2-1 & d^2-1 & 0 &-d^2+1 \\
   d+1 &-1 &-d^2+1 & d
 \end{pmatrix}$,

F-matrices:  see below. 

Realizations: $(A_1,7)_{\frac{1}{2}}$, complex conjugate of
$(G_2)_2$.

The list of all $F$ matrices for an MTC can occupy many pages. But
they are needed for the computation of quantum invariants using
graph recouplings, the Hamiltonian formulation of MTCs as in
\cite{LW} or the study of anyon chains \cite{FTL}. For the MTCs
$(A_1,k)_{\frac{1}{2}}$ with odd $k$, all the data of the theory
can be obtained from \cite{KL}. For $k=5$, choose
$A=ie^{-\frac{2\pi i}{28}}$, the label set is $\mL=\{0,2,4\}$ in
\cite{KL} and $0=1,4=\alpha,2=\beta$. For $k=7$, set
$A=ie^{\frac{2\pi i}{36}}$, the label set is $\mL=\{0,2,4,6\}$ in
\cite{KL} and $0=1,6=\alpha,2=\omega,4=\rho$. The twist is given
by $\theta_a=(-1)^a A^{a(a+2)}$, and the braiding
$R_c^{ab}=(-1)^{\frac{a+b-c}{2}}A^{-\frac{a(a+2)+b(b+2)-c(c+2)}{2}}$.
The formulas for $6j$ symbols can be found in Chapter 10 of
\cite{KL}.  The $F$ matrices from \cite{KL} are not unitary, but
the complete data can be presented over an abelian Galois
extension of $\Q$.  To have unitary $F$ matrices, we need to
normalize the $\theta$ symbols as
$\theta(i,j,k)=\sqrt{d_id_jd_k}$.

The $(A_1,k)_{\frac{1}{2}}$, $k$ odd, MTCs have peculiar
properties regarding the relation between the bulk $(2+1)$-TQFTs
and the boundary RCFTs.  To realize $(A_1,k)_{\frac{1}{2}}$, $k$
odd, using the Kauffman bracket formalism, we set $A=ie^{\pm
\frac{2\pi i}{4(k+2)}}$. In order to
follow the convention above, we choose
$A=ie^{-\frac{2\pi i}{4(k+2)}}$ if $k=1$ mod $4$, and
$A=ie^{\frac{2\pi i}{4(k+2)}}$ if $k=-1$ mod $4$.  Note that in
both cases $A$ is a $2(k+2)$th root of unity.  We have

\begin{equation}\label{decom1}
(A_1,k)=\overline{ (A_1,k)_{\frac{1}{2}} }\times \textrm{the
semion},
\end{equation}
if $k=1$ mod $4$, and

\begin{equation}\label{decom2}
(A_1,k)=(A_1,k)_{\frac{1}{2}}\times \overline{\textrm{the
semion}},
\end{equation}
if $k=-1$ mod $4$.  The central charge of $(A_1,k)$ is
$\frac{3k}{k+2}$, which implies that the central charge of
$(A_1,k)_{\frac{1}{2}}$ is $c_k=1-\frac{3k}{k+2}$ if $k=1$ mod
$4$, and $c_k=1+\frac{3k}{k+2}$ if $k=-1$ mod $4$.

In Table \ref{rank12}, we list all unitary quantum groups categories of rank$\leq 12$ from the
standard construction.  For notation, see \cite{HRW}.

\begin{remark}  The following serves as a guide to Table \ref{rank12}.

\begin{enumerate}
\item In general we will list these categories as $(X_r,k)$ for the category obtained from
a quantum group of type $X_r$ at level $k$.  Observe that the corresponding root of unity
is of order $\ell=mk+h$ where $m=1$ for $X=A,D$, or $E$; $m=2$ for $X=B,C$ or $F$ and $m=3$ for $X=G$, and $h$ is the dual Coxeter number.


\item The category $(A_r,k)$ has a modular subcategory $(A_r,k)_{\frac{1}{r+1}}$ generated by the objects with integer weights provided $\gcd(r+1,k)=1$.  These are found on line 5 of Table \ref{rank12} where $$L=\{(1,2s+1),(2s,2),(2,4),(2,5),(2,7),(3,3),(4,3),(6,3): 1\leq s\leq 11\}.$$

\item We include the examples of pseudo-unitary categories coming from low-rank coincidences for quantum groups of types $F_4$ and $G_2$ at roots of unity of order coprime to $2$ and $3$ respectively.
\item This list includes different realizations of equivalent categories.  We eliminate those coincidences that occur because of Lie algebra isomorphisms
such as $\mathfrak{sp}_4\cong \mathfrak{so}_5$ etc., and do not include the trivial rank$=1$ category.
        \item NSD means the category contains non-self-dual objects.
        \item  ``c.f. $(X_r,k)$" means the
categories in question has the same fusion rules as those of $(X_r,k)$.
\item We include the three categories
coming from doubles of finite groups with rank$\leq$12, although they are not strictly speaking of quantum group type.
\end{enumerate}
\end{remark}

\begin{table}\caption{Unitary Quantum Group Categories of rank $\le 12$}\label{rank12}
\begin{tabular}{*{4}{|c}|}
\hline
$(X_r,k)$ & Rank & Notes & $\ell$\\
\hline\hline
 $(A_r,1)$, $r\leq 11$ & $r+1$  & $r\ge 2$ NSD, abelian & $r+2$\\
\hline $(A_1,k)$, $k\leq 11$ & $k+1$ &  & $k+2$\\
\hline $(A_2,2)$ & $6$ & NSD & $5$\\
\hline $(A_2,3),(A_3,2)$ & $10$ & NSD & $6$\\
\hline $(A_r,k)_{\frac{1}{r+1}}$, $(r,k)\in L$& $\frac{1}{r+1}\binom{k+r}{k}$ & $r\ge 2$ NSD  & $k+r+1$
\\ \hline  $(B_r,1)$ & $3$ &  c.f. $(A_1,2)$ & $4r$
\\ \hline  $(B_r,2)$, $r\leq 8$ & $r+4$  & finite braid image? & $4r+2$
\\ \hline  $(B_2,3)$& $10$  &  & $12$
\\ \hline  $(C_r,1)$ $r\leq 11$ & $r+1$  & c.f. $(A_1,r)$ &  $2(r+2)$
\\ \hline  $(C_3,2)$ & $10$  & & $12$
\\ \hline  $(D_{2r},1)$ & $4$ & $r$ even c.f. $D^\omega(\Z_2)$ & $4r-1$
\\ \hline  $(D_{2r+1},1)$ & $4$ &  c.f. $(A_3,1)$ & $4r+1$
\\ \hline  $(D_r,2)$, $r=4,5$& $11,12$  & $r=5$ NSD & $8,10$
\\ \hline  $(E_6,k)$, $k=1,2$ & $3,9$ & NSD & $13,14$
\\ \hline  $(E_7,k)$, $1\leq k\leq 3$& $2,6,11$& & $19,20,21$
\\ \hline  $(E_8,k)$, $2\leq k\leq 4$& $3,5,10$ & & $32,33,34$
\\ \hline  $(F_4,k)$, $1\leq k\leq 3$& $2,5,9$ & & $20,22,24$
\\ \hline  $(G_2,k)$, $1\leq k\leq 5$& $2,4,6,9,12$ & & $15,18,21,24,27$
\\ \hline $F_4$ & $10$ & c.f. $(E_8,4)$ & $17$
\\ \hline $G_2$ & $5,8,10$ &  & $11,13,14$
\\ \hline $D^\omega(\Z_2)$ & $4$ & prime &
\\ \hline $D^\omega(\Z_3)$ & $9$ & prime &
\\ \hline $D^\omega(S_3)$ & $8$ & c.f. $(B_4,2)$ &
\\
\hline
\end{tabular}
\end{table}

\subsection{Classification}\label{count}

In this section, we explain Table \ref{rank4table}.  We identify
MTCs whose label sets differ by permutations.  For the trivial
MTCs, the two MTCs are distinguished by the $S$ matrices: $S=(\pm
1)$.

For the $\Z_2$ fusion rule, unitary MTCs are the semion MTC and
those from the two symmetries $S\rightarrow -S$ and complex
conjugate.

For the Fibonacci fusion rule, unitary MTCs are the Fibonacci MTC
and those from the two symmetries $S\rightarrow -S$ and complex
conjugate.

For the $\Z_3$ fusion rule, all unitary MTCs are the one listed in
last subsection and those from the two symmetries $S\rightarrow
-S$ and complex conjugate.

For the Ising fusion rule, there are a total of $16$ theories
divided into two groups according to the Frobenius-Schur indicator
of the non-abelian anyon $X,X^2=1+Y$.  There are $8$ unitary MTCs
with Frobenius-Schur indicator=$1$.  Their twists are given by
$\theta_X=e^{\frac{m\pi i}{8}}$ for $m=1,7,9,15$.  The Ising MTC
is the simplest one with $m=1$ and central charge $c=\frac{1}{2}$.
The theory $m=1,m=15$ are complex conjugate of each other, so are
the $m=7,9$.  The other $4$ MTCs are obtained by choosing $-S$.
There are $8$ unitary MTCs with Frobenius-Schur indicator=$-1$.
Their twists are $\theta_X=e^{\frac{m\pi i}{8}}$ for
$m=3,5,11,13$.  The $SU(2)$ at level $k=2$ is the simplest one
with $m=3$ and central charge $c=\frac{3}{2}$.  The MTCs $m=3$ and
$m=13$ are complex conjugate, so are $m=5,11$.  The other $4$ are
those with $-S$. The Ising MTC is not an $SU(2)$ theory.  It can,
however, be obtained as a quantum group category as the complex
conjugate of $E_8$ at level$=$2. Note that
the $F$ matrices in each group of $8$ are the same, but their
braidings are different.  The $SU(2)$ level=$2$ theory has
$F_{X}^{XXX}=-F_{\sigma}^{\sigma\sigma\sigma}$ with the other $F$
matrices the same as the Ising theory.

For the $(A_1,5)_{\frac{1}{2}}$ fusion rule, all unitary MTCs are
the one listed in last subsection and those from the two
symmetries $S\rightarrow -S$ and complex conjugate.

For the $\Z_2\times \Z_2$ fusion rules, there are two groups of
theories depending whether or not the theory is a product.  There
are $4$ theories which are not direct products, and $6$ product
theories.  The toric code MTC has another version, which could
also be called the toric code: it has $\theta_e=\theta_m=-1$. All
$F$ matrices are $1$.  The braidings
$R_{1}^{ee}=R_{1}^{mm}=R_{\e}^{em}=-1,R_{\e}^{me}=1$, and others
are the same as the toric code.  Another two are the $-S$
versions.  The product theories are the products of the semion MTC
and its complex conjugate.  There are $4$ possible theories, but
two of them are the same: semion $\times$ complex conjugate is the
same as complex conjugate $\times$ semion.  Hence there are $3$
theories here.  With the $-S$ versions, we have $6$ product
theories.

For the $\Z_4$ fusion rule, the Galois group action of the MTC
listed above is $\Z_4$.  Its actions give rise to $4$ theories
with $\theta_{X}=\theta_{X^*}=e^{\frac{m\pi i}{4}}$ for
$m=1,3,5,7$.  They all have the same $F$ matrices.  The $-S$
versions give a total of $8$.

For the $(A_1,3)$ fusion rule, this is the product of the semion
fusion rule with the Fibonacci fusion rule.  There are $4$ product
theories from semion, Fibonacci and their complex conjugates.
These $4$ theories are different, and the other $4$ come from
their $-S$ versions.  Let us choose the product of the semion with
the Fibonacci as a representative theory, then we have $4$ anyons,
$1,\varphi,\tau,s$, where $\tau$ is the Fibonacci anyon, and
$\varphi$ is the same as $\tau$ tensoring the semion $s$.

For the $(A_1,7)_{\frac{1}{2}}$ fusion rule, all unitary MTCs are
the one listed in last subsection and those from the two
symmetries $S\rightarrow -S$ and complex conjugate.

The analysis of the Fibonacci $\times$ Fibonacci fusion rule is
the same as that of the semion $\times$ semion fusion rule.

\section{Conjectures and Further Results}\label{conjecture}

In this section we briefly discuss several conjectures concerning the
structure and application of MTCs.

\subsection{Fusion Rules and the Finiteness Conjecture}

Since topological phases of matter are discrete in
the space of theories, therefore, MTCs,  encoding the universal properties of
topological phases of matter, should also be discrete.

 It is conjectured \cite{Wang}:
\begin{conj}
 If the rank of MTCs is fixed, then there are
only finitely many equivalence classes of MTCs.
\end{conj}
 By Ocneanu rigidity, this is equivalent
to there are only finitely many modular fusion rules realizing by
MTCs of a fixed rank.

\begin{prop}

There are only finitely many equivalence classes of unitary MTCs
with total quantum order $D\leq  c$, where $c$ is any given
universal constant.

\end{prop}

\begin{proof}

For a unitary rank=$n$ MTC, all quantum dimensions $d_r\geq 1,
r\in \mL$. So $D\geq \sqrt{n}$.  If $D\leq c$, then $n\leq c^2$.
By Verlinde formula \ref{verlindeformula}, we have
$n_{i,j}^k=|\sum_{r=0}^{n-1}
\frac{s_{ir}s_{jr}s^{*}_{kr}}{s_{0r}}|\leq
D\sum_{r=0}^{n-1}\frac{1}{d_r}\leq nD\leq c^3$ for any $i,j,k$.
Therefore, there are only finitely many possible fusion rules.  By
Ocneanu rigidity, there are only finitely many possible MTCs.

\end{proof}

\subsection{Topological Qubit Liquids and the Fault-tolerance Conjecture}

 Topological phases of matter are quantum liquids such as the
 electron liquids exhibiting the FQHE, whose topological
 properties emerged from microscopic degrees of freedom.  This
 inspires the following discussion.

 Let $\Delta$ be a triangulation of a closed surface $\Sigma$,
 $\Gamma_{\Delta}$ be its dual triangulation: vertices are centers
 of the triangles in $\Delta$, and two vertices are connected by
 an edge if and only if the corresponding triangles of $\Delta$
 share an edge.  The dual triangulation $\Gamma_{\Delta}$ of
 $\Delta$ is a celluation of $\Sigma$ whose $1$-skeleton is a
 tri-valent graph.  It is well-known that any two triangulations
 of the same surface $\Sigma$ can be transformed from one to the other by
 a finite sequence of two moves and their inverses: the
 subdivision of a triangle into $3$ new triangles; and
 the diagonal flip of two adjacent triangles that share an edge
 (=the diagonal).  Dualizing the triangulations into celluations, the two
 moves become the inflation of a vertex to a triangle and the $F$ move.

 \begin{definition}

 \begin{enumerate}

 \item Given an integer $k>0$, a $k$-local, or just local, qubit model on
 $(\Sigma,\Gamma_{\Delta})$ is a pair $(\mH_{\Delta},H_{\Delta})$,
 where $\mH_{\Delta}$ is the Hilbert space
 $\otimes_{e\in \Gamma_{\Delta}} \C^2$, and $H_{\Delta}$ is a $k$ local Hamiltonian
  in the following sense: $H_{\Delta}$ is a sum of Hermitian operators of the form
 $id \otimes \cdots \otimes id \otimes O_k \otimes id \otimes \cdots
 \otimes id$, where $O_k$ acts on $\leq k$ qubits.

 \item A modular functor $V$ is realized by a topological qubit
 liquid if there is a sequence of triangulations
 $\{\Delta_i\}_{i=1}^{\infty}$ of
 $\Sigma$ whose meshes $\rightarrow 0$ as $i\rightarrow \infty$, an integer $k$, and uniform
 local qubit models on $(\Sigma,\Gamma_{\Delta_i})$ such that

 (i) the groundstates manifold of each $H_{\Delta_i}$ is canonically isomorphic
 to the modular functor $V(\Sigma)$ as Hilbert spaces;

 (ii) the mapping class group acts as unitary transformations compatibly;

 (iii) there is a spectral gap in the following sense: if the
 eigenvalues of the Hamiltonians $H_{\Delta_i}$ are normalized such that $0=\lambda_0^i<
 \lambda_1^i < \cdots$, then $\lambda_1^i \geq c$ for all $i$,
 where $c>0$ is some universal constant.

 \end{enumerate}

 \end{definition}

 The scheme for the
 local qubit models should be independent
 of the geometry of the surface $\Sigma$, and have a uniform local description.
 The modular functor determines
 a unique topological inner product on $V(\Sigma)$.  We require
 that the restricted inner products from
 $\mH_{\Delta_i}$ to the groundstates of $H_{\Delta_i}$ agree with the topological inner product on $V(\Sigma)$.
 To identify the
 Hilbert space $\mH_{\Delta_i}$ of one triangulation with another,
 we consider the two basic moves: $F$ move and inflation of a vertex.
  The $F$ move does not change the number of qubits, so the two
  Hilbert spaces $\mH_{\Delta_i}$ have the same number of qubits.  We require that the
  identification to be an isometric.  For the inflation of a vertex,
  the inflated celluation has $3$ new qubits, so we need to choose a homothetic
  embedding with a universal homothecy constant.

The action of the mapping class group is defined as follows:
consider the moduli space of all triangulations of $\Sigma$ that
two triangulations are equivalent if there dual graphs
$\Gamma_{\Delta}$ are isomorphic as abstract graphs. By a sequence
of diagonal flips, we can realize a Dehn twist. Each diagonal flip
is an $F$ move,  and their composition is the unitary
transformation associated to the Dehn twist.

\begin{conj}

\begin{enumerate}

\item Every doubled MTC $\mC$ can be realized as a topological qubit liquid.

\item The groundstates $V(\Sigma)\cong H_{\Delta_i}\subset
\mH_{\Delta_i}$ form an error-correction code for each
triangulation $\Delta_i$.

\end{enumerate}

 \end{conj}

\subsection{Topological Quantum Compiling and the Universality Conjecture}\label{compiling}

Every unitary MTC gives rise to anyonic models of quantum
computers as in \cite{FKLW}.  Quantum gates are realized by the
braiding matrices of anyons, i.e. the afforded representations of
the braid groups. Topological quantum compiling is the question of
realizing desired unitary transformations by braiding matrices in
quantum algorithms, in particular for those algorithms which are
first described in the quantum circuit model such as Shor's famous
factoring algorithm.

To choose a computational subspace, we will use the
so-called conformal block basis for the Hilbert space $V(D^2,a_i;a_{\infty})$ of a punctured disk, where $a_{\infty}$ labels
the boundary.  Conformal block basis is in one-one correspondence to admissible
labelings $m,n,\cdots, p$ of the internal edges of the following
graph subject to the fusion rules at each trivalent vertex.  As
explained in Section \ref{fmatrix}, the tri-valent vertices
also need to be indexed if multiplicities $n_{i,j}^k>1$.

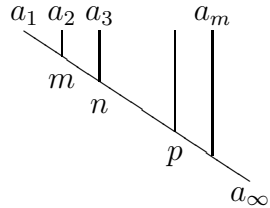
\[ \xy
(0,10)*{}="A1"; (5,10)*{}="A2";(10,10)*{}="A3"; (25,10)*{}="Am";
(30,-10)*{}="AI";(5,6.6)*{}="M"; (10,3.3)*{}="N";
(20,-3.3)*{}="P"; (25,-6.6)*{}="I"; "A1";"AI" **\dir{-}; "A2";"M"
**\dir{-};"A3";"N" **\dir{-}; (20,10)*{}="P'"; "P'";"P" **\dir{-};
"Am";"I" **\dir{-}; (0,12)*{a_1}; (5,12)*{a_2};(10,12)*{a_3};
(30,-12)*{a_{\infty}}; (10,0)*{n};(5,3.3)*{m}; (20,-6.6)*{p};
(25,12)*{a_m};
\endxy \]

The braiding of two anyons $a_i,a_{i+1}$ in a conformal block
basis state is represented by the stacking the braid on top of the
above graph at $i,i+1$ positions.

\begin{definition}
An MTC $\CC$ has \emph{property F} if for every object $X$ in
$\CC$ and every $m$ the representation $\rho_X^m$ of $\B_m$ on
$V(D^2,X,\cdots,X;a_{\infty})$ factors over a finite group for any
$a_{\infty}\in \mL$.
\end{definition}

The following is conjectured by the first author (see \cite{propF}):
\begin{conj}
Let $\CC$ be an MTC.
\begin{enumerate}
\item[(a)] If $\CC$ is unitary, then it has property \emph{F} if
and only if $(d_i)^2\in\N$ for each simple object $X_i$ or,
equivalently, if and only if the global quantum dimension
$D^2\in\N$. \item[(b)] In general, $\CC$ has property \emph{F} if
and only if $(\FPdim(X_i))^2\in\N$ for each simple object $X_i$,
where $\FPdim$ is the Frobenius-Perron dimension, i.e. the
Frobenius-Perron eigenvalue of the fusion matrix $N_i$.
\end{enumerate}
\end{conj}

The verification of this conjecture for UMTCs of rank$\leq 4$ is summarized in Table \ref{rank4}.

\begin{table}\caption{Unitary prime MTCs rank$\leq$4}\label{rank4}
\begin{tabular}{|c|p{4cm}|c|c|}
\hline
Realization & $\PSL(2,\Z)$, Relations & Property $F$?& Universal Anyons\\
\hline\hline
$\Vect_\C$ & $1$, $S=T=1$ &Yes& \\
\hline
 $(A_1,1)$ &\raggedright $\PSL(2,3)$, $T^4=I$ &Yes &\\
\hline $(A_1,3)_{\frac{1}{2}}$ & $\PSL(2,5)$, $T^5=I$ &No & $\tau$\\
\hline $(A_2,1)$ & $\PSL(2,3)$, $T^3=I$ &Yes &\\
\hline $(A_1,2)$ &\raggedright $\PSL(2,8)$, \\$T^{16}=(T^2S T)^3=I$ &Yes &\\
\hline $(A_1,5)_{\frac{1}{2}}$ &\raggedright $\PSL(2,7)$,\\ $T^7=(T^4S T^4S)^2=I$ &No & $\alpha$,$\beta$\\
\hline $(A_3,1)$ &\raggedright $\PSL(2,8)$, \\$T^8=(T^2S T)^3=I$ &Yes &\\
\hline $D(\Z_2)$ & $\PSL(2,2)$, $T^2=I$ &Yes &\\
\hline $(A_1,7)_{\frac{1}{2}}$ &\raggedright $\PSL(2,9)$, \\$T^9=(T^4S T^5S)^2=I$ &No &$\alpha,\omega,\rho$\\
\hline
\end{tabular}
\end{table}

\begin{theorem}\label{universalanyon}

The following anyons are universal in the sense of \cite{FKLW}:

the Fibonacci anyon $\tau$, the $(A_1,5)_{\frac{1}{2}}$ anyons
$\alpha,\beta$, the $(A_1,7)_{\frac{1}{2}}$ anyons
$\alpha,\omega,\rho$, the two anyons $\varphi, \tau$ in $(A_1,3)$
(see \ref{count} for notation), and the two $\tau$'s in $Fib\times
Fib$.

Universal anyonic quantum computation can also be achieved with
the anyon $\tau \times \tau$ in $Fib\times Fib$, but images of the
representations of the braid groups from this anyon are not as
large as possible.

Anyons that correspond to $\varphi, \tau,\alpha,\beta,\omega,\rho$
in other versions are also universal.

\end{theorem}

\begin{proof}

We deduce the proof from \cite{FLW}\cite{LRW}\cite{LWa}.

Universality of $\varphi$ and $\tau$ is given in \cite{FLW}.  The anyons $\alpha$ are both the
fundamental representations of $(A_1,k)$ up to abelian anyons. The
universality of fundamental representation anyons are established in \cite{FLW}.
Therefore, both $\alpha$'s are universal.

To prove that $\beta, \omega,\rho$ are universal, we first show
that their braid representations are irreducible.  By inspecting
the braiding eigenvalues in Section \ref{explicitdata}, we see
that they satisfy the conditions of \cite{TuWe}[Lemma 5.5]
\cite{HRW}[Proposition 6.1].  It follows that the braid
representations are irreducible.  Universality now can be proved
following \cite{FLW} or \cite{LRW}.

\end{proof}

\appendix

\section{Non-self dual rank$\leq 4$ MTCs with S.
Belinschi}\label{nonselfdual}

Every rank=$1,2$ MTC is self-dual, so we will start with rank=$3$.

\subsection{Nonselfdual Rank=3}

The three labels will be $0,1,2$ such that $\hat{0}=0,\hat{1}=2,
\hat{2}=1$.  The modular $\tilde{S}$ matrix is of the form:
\[
 \begin{pmatrix}
   1& d&d \\
   d& x &\bar{x}\\
   d&\bar{x} &x
 \end{pmatrix}.
 \]
 $\tilde{s}_{22}=\tilde{s}_{11},\ts_{12}=\overline{\ts_{11}}$ follows from
 $\tilde{s}_{\hat{i},j}=\overline{\ts_{i,j}}$.
 Unitarity of $S$ implies
 \begin{equation}\label{orth1}
 1+d^2=2|x|^2,
 \end{equation}
 \begin{equation}\label{orth2}
 d^2+x^2+\bar{x}^2=0,
 \end{equation}
 \begin{equation}\label{orth3}
 1+x+\bar{x}=0.
 \end{equation}

 The fusion matrix $N_1$ has eigenvalues
 $d,\frac{x}{d},\frac{\bar{x}}{d}$.  Their sum $d+\frac{x+\bar{x}}{d}=d- \frac{1}{d}$ is an
 integer.  Their product $\frac{|x|^2}{d}=\frac{1+d^2}{2d}=\frac{1}{2}(\frac{1}{d}+d)$ is also an integer.
 Therefore, $d$ is an integral multiple of $\frac{1}{2}$, so
 $d$ is an integer.

 Let $\theta$ be the twist of label $1$, hence of label $2$.  Using identity
 (\ref{balance}), we get
 \begin{equation}
 1-2d^2+\theta+\theta^{-1}=0.
 \end{equation}
 Therefore, $2d^2\leq 3$.  Since $d\neq 0$, the only possible
 integers are
$d^2=1$, hence $|x|=1$.  Then $1+x+\bar{x}=0$ leads to $x=e^{\pm
\frac{2\pi i}{3}}.$

\subsection{Nonselfdual rank=4}

Now we turn to the non-self dual rank=$4$ case.  The $4$ labels
will be denoted as $1,Y,X,X^*$, where $Y$ is self dual and $X,X^*$
dual to each other.  Taking into account of all symmetries among
$n_{i,j}^k$, we can write the non-trivial fusion matrices as:

$
 N_Y=\begin{pmatrix}
   0&1& 0&0 \\
   1&n_1& n_2 &n_2\\
   0 &n_2&n_3&n_4\\
   0 &n_2&n_4&n_3
 \end{pmatrix}$;

$N_X=\begin{pmatrix}
   0&0& 1&0 \\
   0&n_2& n_3 &n_4\\
   0 &n_4&n_5&n_6\\
   1 &n_3&n_7&n_7
 \end{pmatrix}$;

 $N_{X^*}=\begin{pmatrix}
   0&0& 0&1 \\
   0&n_2& n_4 &n_3\\
   1 &n_3&n_7&n_7\\
   0 &n_4&n_6&n_5
 \end{pmatrix}$.

 The modular $\tS$ matrix is of the form:

$ \tS=\begin{pmatrix}
   1&d_1& d_2&d_2 \\
   d_1&x& y &y\\
   d_2& y& z&\bar{z}\\
   d_2 &y &\bar{z}& z
 \end{pmatrix},$ where $x,y$ are real, and $z$ is not real.

We will work on unitary modular symbols, so $d_1\geq 1, d_2\geq 1$.
The argument for general case should have only minor changes.

 The identity $N_XN_Y=N_YN_X$ leads to the identities:
 \begin{equation}
 1+n_1n_3+n_2(n_5+n_7)=n_2^2+n_3^2+n_4^2,
 \end{equation}
  \begin{equation}
 n_1n_4+n_2(n_6+n_7)=n_2^2+2n_3n_4,
 \end{equation} \begin{equation}
 n_1n_4+n_2(n_5+n_6)=n_2^2+2n_3n_4.
 \end{equation}
 $N_YN_{X^*}=N_{X^*}N_Y$ gives no new identities.  But
 $N_XN_{X^*}=N_{X^*}N_X$ gives us:
  \begin{equation}\label{n2n4dependence}
 n_2n_4+n_4n_6=n_2n_3+n_4n_5,
 \end{equation}
 \begin{equation}
 n_5=n_7,
 \end{equation} \begin{equation}
 n_4^2+n_6^2=1+n_3^2+n_7^2.
 \end{equation}

 Case 1: $n_4=0$.

 If $n_4=0$, then $n_2n_3=0$.  First if $n_2=0$, then
 $1+n_1n_3=n_3^2$ which implies $n_3=1,n_1=0$.  It follows that
 $n_1=n_2=n_4=0,n_3=1$.  This leads to $n_6^2=2+n_7^2$ which has
 no solutions.  Secondly if $n_3=0$, then $n_6^2=1+n_7^2$ which
 implies $n_6=1,n_7=0$.  Hence $n_3=n_4=n_5=n_7=0, n_6=1$.  This
 leads to $n_2=1$, and $n_1$ is arbitrary.
To rule out this case, notice that the labels $1,X,X^*$
 have exactly the same fusion rules as the rank=$3$ non-self dual
 theory.  Therefore, it is a pre-modular category with the same
 fusion rules, which is necessarily modular by \cite{Br}:
 Suppose otherwise, then $(d_2,z,\bar{z})$ would be a
 $d_2$ times $(1,d_2,d_2)$ as vectors, contradicting $z$ is not real.
 It follows $d_2=1, z=\omega$ for some $\omega^3=1$.  Comparing
 the squared lengths of row $1$ and row $3$ of the $\tS$ matrix, we see that $y^2=d_1^2$.  Also
 note that $d_1^2=3+n_1d_1$.  Equality of the squared lengths of row
 $1$ and row $2$ implies $x^2+2d_1^2=3$.  Since $x$ is real, this
 does not hold if $d_1>0$.

 Case 2: $n_4\neq 0$.

 If $n_2=0$, then $n_1=2n_3, 1+n_1n_3=n_3^2+n_4^2.$  Hence
 $1+n_3^2=n_4^2$ which implies $n_4=1,n_3=0$.  So we have
 $n_1=n_2=n_3=0, n_4=1,n_5=n_6=n_7$.  The labels $1,Y$ form a
 subcategory same as the $\Z_2$ theory, hence $d_1^2=1,x^2=1$.
 If $x=-1$, then $y=0$, and $d_1d_2=0$ which is a contradiction.
 If $x=1$, then $y^2=d_2^2$.  So $d_2^2=d_1^2=1$.
 Using $d_2^2=1+n_3d_1+2n_5d_2$ below, we see that $n_3,n_5$, hence
 $n_6=n_7=0$.  So we have $n_1=n_2=n_3=n_5=n_6=n_7=0, n_4=1$,
 which is the $\Z_4$ fusion rule.

 Suppose $n_4\neq 0, n_2\neq 0$.

 The fusion rules in Table \ref{rank4} gives us the following identities:
 \begin{equation}
 d_1^2=1+n_1d_1+2n_2d_2,
 \end{equation}
 \begin{equation}\label{d1d2product}
 d_1d_2=n_2d_1+(n_3+n_4)d_2,
 \end{equation}
 \begin{equation}\label{d2square}
 d_2^2=n_4d_1+(n_5+n_6)d_2,
 \end{equation}
 \begin{equation}
 d_2^2=1+n_3d_1+2n_5d_2.
 \end{equation}

 Combining equations, we have
 \begin{equation}\label{d1d2dependence}
 (n_4-n_3)d_1+(n_6-n_5)d_2=1.
 \end{equation}

 If $n_4=n_3$, then $n_6=1+n_7^2$ which implies $n_6=1,n_7=0$.
 Hence $n_5=0$.  By equation (\ref{d1d2dependence}), $n_5=n_6$ which is a contradiction.

 If $n_5=n_6$, then $n_4^2=1+n_3^2$ which implies $n_3=0,n_4=1$.
 Solving all equations, we get $n_1=n_2=n_3=0,n_4=1,n_5=n_6=n_7$,
 which is the $\Z_4$ fusion rule.

 So we may assume from now on $n_2\neq 0, n_4\neq 0, n_4\neq n_3,
 n_5\neq n_6$.
 By equation (\ref{n2n4dependence}), we have
 \begin{equation}
 n_4(n_5-n_6)=n_2(n_4-n_3).
 \end{equation}
 Hence we have
 \begin{equation}
 d_2=\frac{n_4}{n_2}d_1-\frac{n_4}{n_2(n_4-n_3)}.
 \end{equation}

Plugging into (\ref{d1d2dependence}) and simplifying, we have
\begin{equation}\label{d1square}
d_1^2=(n_1+2n_4)d_1-\frac{n_3}{n_4-n_3}.
\end{equation}

The orthogonality of $\tS$ gives us:

\begin{equation}\label{xysquare}
x^2+2y^2=1+2d_2^2,
\end{equation}
\begin{equation}
y^2+2|z|^2=1+d_1^2+d_2^2,
\end{equation}
\begin{equation}\label{xd}
(1+x)d_1+2yd_2=0,
\end{equation}
\begin{equation}\label{d1d2z}
yd_1+(1+z+\bar{z})d_2=0,
\end{equation}
\begin{equation}\label{zz}
d_1d_2+(x+z+\bar{z})y=0,
\end{equation}
\begin{equation}
d_2^2+y^2+z^2+\bar{z}^2=0.
\end{equation}

Note that $y$ cannot be $0$.  Suppose otherwise, then $x=-1$, so $d_2=0$, a contradiction.

The eigenvalues of $N_Y$ are $d_1,
\frac{x}{d_1},\frac{y}{d_2},\frac{y}{d_2}$.  Their sum
$d_1+\frac{x}{d_1}-\frac{(1+x)d_1}{d_2^2}=d_1-\frac{d_1}{d_2^2}+(\frac{1}{d_1}-\frac{d_1}{d_2^2})x$
is an integer. The eigenvalues of $N_X$ are $d_2,
\frac{y}{d_1},\frac{z}{d_2},\frac{\bar{z}}{d_2}$.  Their sum
$d_2+\frac{y}{d_1}+\frac{z+\bar{z}}{d_2}$ is an integer.

If $\frac{1}{d_1}-\frac{d_1}{d_2^2}= 0$, then $d_1^2=d_2^2$. By
equation (\ref{d1d2product}), $\pm d_1=n_2+n_3+n_4$, then $d_1,d_2$
are integers. But the sum of the eigenvalues of $N_Y$
$d_1-\frac{1}{d_1}$ is also an integer, so $d_1=\pm 1$. It follows
that $\pm 1=n_2+n_3+n_4$, but $n_2,n_4$ are both $\neq 0$, a
contradiction.

If $\frac{1}{d_1}-\frac{d_1}{d_2^2}\neq 0$, then $x$ and
subsequently all $y,z+\bar{z},|z|^2$ are in $\Q(d_1,d_2)$.
So all $x,y,z+\bar{z},|z|^2,z^2+\bar{z}^2$ are in $\Q(d_1,d_2)$.
By equation \ref{d1d2dependence}, $\Q(d_1,d_2)$ is a degree$\leq 2$ Galois extension of $\Q$.
Therefore, the Galois group of the characteristic polynomial
$p_1(t)$ of $N_Y$ is either trivial or $\Z_2$.  If it is trivial,
then all eigenvalues $d_1,
\frac{x}{d_1},\frac{y}{d_2},\frac{y}{d_2}$ and $d_2$ are integers.
So $d_1,d_2,x,y$ are all integers.

From the unitary assumption, $d_1, d_2\geq 1$.  Since $\frac{x}{d_1},
\frac{y}{d_2}$ are integers, $|x|\geq d_1,|y|\geq d_2$.  Equation (\ref{xysquare})
implies $x=\pm 1, y=\pm d_2$.  Since $\frac{x}{d_1}$ is an integer, $|x|=d_1=1$.  Then
$2yd_2=-2$ implies $d_2=1$,
contradicting $\frac{1}{d_1}-\frac{d_1}{d_2^2}\neq 0$.

Therefore the Galois group of $p_1(t)$ is $\Z_2$.  Since $p_1(t)$
has a pair of repeated roots, then $p_1(t)$ is $(t-m)^2 q_1(t)$
for some irreducible quadratics $q_1(t)$ and integer $m$ or
$(q_1(t))^2$.  Assume $q_1(t)=t^2+bt+c$, where $b,c$ are integers.
 Note that $d_1$ has to be an irrational root of $p_1(t)$.
If $p_1(t)$ has integral roots $m$, then $\frac{y}{d_2}=m$, so
$y^2\geq d_2^2$.  $x=d_1\frac{x}{d_1}=c$ implies $|x|\geq 1$.  By equation (\ref{xysquare}), $y^2\geq d_2^2$ implies
$x^2\leq 1$, hence
$|x|=1,y^2=d_2^2$.  It follows from equation (\ref{xd}) that $d_1=d_2^2$.
By equation (\ref{d2square}), $(n_4-1)d_1+(n_5+n_6)d_2=0$.  Since $n_4\geq 1$, it follows that
$n_4=1,n_5=n_6=0$, contradicting $n_5\neq n_6$.

Hence $p_1(t)=q_1(t)^2$, and $d_1=\frac{x}{d_1}$, i.e.
$x=d_1^2\geq 1$, and $y^2\leq d_2^2$. So the roots of $p_1(t)$ are
$d_1,d_1,\frac{y}{d_2},\frac{y}{d_2}$. Then $d_1+\frac{y}{d_2}$
and $\frac{d_1y}{d_2}$ are both integers. By equations
\ref{d1d2z},\ref{zz},
$\frac{d_1d_2}{y}+x=-(z+\bar{z})=\frac{yd_1}{d_2}+1$ is an
integer.  On the other hand,
$\frac{d_1d_2}{y}+x=x(\frac{d_2}{d_1y}+1)$, so $x=d_1^2$ would be
a rational number $s$ if $\frac{d_2}{d_1y}+1\neq 0$. Then
$d_1=\sqrt{s}$, which is also $\frac{-b\pm \sqrt{b^2-4c}}{2}$, but
not a rational number, hence $b=0$, a contradiction.  If
$\frac{d_2}{d_1y}+1=0$, then $y=-\frac{d_2}{d_1}$.  Substituting
this and $x=d_1^2$ into equation (\ref{xd}), we get
$d_1^2=2\frac{d_2^2}{d_1^2}-1$.  By equation (\ref{d1d2dependence}),
$\frac{d_2}{d_1}\in \Q$, hence $d_1^2$ would be a rational number
$s$ again, a contradiction.

Putting everything together, we have the only desired modular $\tS$ matrix.

\end{document}